\documentclass[10pt,twoside]{amsart}
\usepackage{latexsym,amssymb,graphicx}
\usepackage{mathrsfs}
\usepackage[all]{xy}

\let\oldtocsection=\tocsection

\let\oldtocsubsection=\tocsubsection

\let\oldtocsubsubsection=\tocsubsubsection

\renewcommand{\tocsection}[2]{\hspace{0em}\oldtocsection{#1}{#2}}
\renewcommand{\tocsubsection}[2]{\hspace{1em}\oldtocsubsection{#1}{#2}}
\renewcommand{\tocsubsubsection}[2]{\hspace{2em}\oldtocsubsubsection{#1}{#2}}

\newtheorem{theorem}{Theorem}[section]
\newtheorem{lemma}[theorem]{Lemma}
\newtheorem{corollary}[theorem]{Corollary}
\newtheorem{proposition}[theorem]{Proposition}

\theoremstyle{definition}

\newtheorem{definition}[theorem]{Definition}
\newtheorem{example}[theorem]{Example}
\newtheorem{construction}[theorem]{Construction}
\newtheorem{remark}[theorem]{Remark}

\newtheorem{notation}[theorem]{Notation}
\numberwithin{equation}{section}

\newcommand{\Z}{\mathbb{Z}}

\newcommand{\C}{\mathcal{C}}

\newcommand{\ffi}{\varphi}
\newcommand{\eps}{\varepsilon}

\newcommand{\St}{\operatorname{St}}

\newcommand{\Hom}{\operatorname{Hom}}

\newcommand{\coker}{\operatorname{coker}}

\newcommand{\SL}{\operatorname{SL}}

\newcommand{\Tens}{\operatorname{Tens}}

\newcommand{\Sp}{\operatorname{Sp}}

\newcommand{\GL}{\operatorname{GL}}

\newcommand{\Skew}{\operatorname{Skew}}
\newcommand{\Pf}{\operatorname{Pf}}

\newcommand{\car}{\operatorname{char}}
\newcommand{\Tot}{\operatorname{Tot}}

\title[Homology of symplectic groups]{The third homology of symplectic groups
	 and algebraic K-theory}
 \author{Marco Schlichting and Husney Parvez Sarwar}

 \address{Marco Schlichting, Mathematics Institute,
Zeeman Building,
University of Warwick,
Coventry CV4 7AL, UK} 

\thanks{Schlichting acknowledges support from EPSRC grant EP/M001113/1 and the Leverhulme Trust}

\email{m.schlichting@warwick.ac.uk}

\address{Husney Parvez Sarwar, 
Department of Mathematics,
Indian Institute of Technology Kharagpur,
Kharagpur - 721302, West Bengal, India} 

\email{mathparvez@gmail.com}

\thanks{Sarwar acknowledges support from the Science and Engineering Research Board (S.E.R.B), India}



\begin{document}
\bibliographystyle{alpha}

\begin{abstract}
We improve the homology stability range for the 3rd integral homology of symplectic groups over commutative local rings with infinite residue field.
As an application, we show that for local commutative rings containing an infinite field of characteristic not $2$ the symbol map from Milnor-Witt $K$-theory to higher Grothendieck-Witt groups is an isomorphism in degrees $\leq 3$.
 \end{abstract}

\maketitle

\tableofcontents

\section{Introduction}
In this paper we improve homology stability ranges for the symplectic groups over commutative local rings with infinite residue field and use this to give explicit presentations of some algebraic $K$-groups.

Recall that for an integer $n\geq 0$ and a commutative ring $R$, the symplectic group $\Sp_{2n}(R)$ is the group of $R$-linear automorphisms of $R^{2n}$ that preserve the standard symplectic inner product 
$\langle x, y \rangle =  \sum_{i=1}^n (x_{2i+1}y_{2i+2} - x_{2i+2}y_{2i+1}),$
 $x,y\in R^{2n}.$
We consider $\Sp_{2n}(R)$ as a subgroup of $\Sp_{2n+2}(R)$ via the embedding $A \mapsto \left(\begin{smallmatrix} 1_{R^2} & 0 \\ 0 & A\end{smallmatrix}\right)$.
The following is part of Theorem \ref{them:H3stability} in the text.
Unless otherwise stated, all homology groups $H_n(G)$ of a group $G$ are taken with integer coefficients.

\begin{theorem}
\label{thmIntro:H3stab}
Let $R$ be a local ring with infinite residue field.
Then inclusion of groups induces a surjection $H_3\Sp_2(R) \twoheadrightarrow H_3\Sp_4(R)$ and for $k\geq 2$ isomorphisms
$H_3\Sp_{2k}(R) \cong H_3\Sp_{2k+2}(R)$:
$$H_3(\Sp_2(R)) \twoheadrightarrow H_3(\Sp_4(R)) \stackrel{\cong}{\longrightarrow} H_3(\Sp_6(R)) \stackrel{\cong}{\longrightarrow} H_3(\Sp_8(R)) \stackrel{\cong}{\longrightarrow} \cdots $$
\end{theorem}

Theorem \ref{thmIntro:H3stab} is optimal in the sense that the first map $H_3\Sp_2(R) \twoheadrightarrow H_3\Sp_4(R)$ is not injective, in general; see Remark \ref{rmk:NotInjective}.
Theorem \ref{thmIntro:H3stab} answers a question raised by Hutchinson and Wendt in \cite[Remark 9.6]{HutchinsonWendt}.
An important consequence of Theorem \ref{thmIntro:H3stab} is the following relative homology stability result proved in Theorem \ref{lem:apl}.

\begin{theorem}
\label{lem:Intro:apl}
	Let $R$ be a local ring with infinite residue field. Then inclusion of groups induces isomorphisms of relative integral homology groups for $i\leq 3$
$$H_i(\SL_3(R),\Sp_2(R)) \stackrel{\cong}{\to}  H_i(\SL_4,\Sp_4)  \stackrel{\cong}{\to} H_i(\SL_6,\Sp_6) \stackrel{\cong}{\to} \cdots  \stackrel{\cong}{\to}  H_i(\SL(R), \Sp(R)).$$
\end{theorem}

The importance of Theorem \ref{lem:Intro:apl} lies in the fact that in degree $3$, the left group $H_3(\SL_3(R),\Sp_2(R))$ was identified in \cite{myEuler} with the third Milnor-Witt $K$-group $K^{MW}_3(R)$ of $R$ and the right group $H_3(\SL(R), \Sp(R))$ is the first non-vanishing homotopy group of the fibre from symplectic $K$-theory to algebraic $K$-theory.
\vspace{1ex}

We generalise Theorem \ref{thmIntro:H3stab} to homology degree $n\geq 3$ by showing the following in Theorem \ref{them:Hkstability} though the stability range here is probably not optimal when $n\geq 4$.

\begin{theorem}
Let $R$ be a local ring with infinite residue field, and let $n\geq 3$ be an integer.
Then inclusion of groups induces a surjection $H_n\Sp_{2n-4}(R) \twoheadrightarrow H_n\Sp_{2n-2}(R)$ and for $k\geq 0$ isomorphisms
$H_n\Sp_{2n+2k-2}(R) \cong H_n\Sp_{2n+2k}(R)$:
$$H_{n}(\Sp_{2n-4}(R)) \twoheadrightarrow H_{n}(\Sp_{2n-2}(R)) \stackrel{\cong}{\longrightarrow}  H_{n}(\Sp_{2n}(R))  \stackrel{\cong}{\longrightarrow} H_{n}(\Sp_{2n+2}(R))  \stackrel{\cong}{\longrightarrow} \cdots
$$
\end{theorem}

This improves on the homology stability ranges for infinite fields $F$ due to Essert \cite{Essert} who proves for $n\geq 0$ the following surjection and isomorphisms
$$H_{n}(\Sp_{2n}(F)) \twoheadrightarrow H_{n}(\Sp_{2n+2}(F)) \stackrel{\cong}{\longrightarrow}  H_{n}(\Sp_{2n+4}(F))  \stackrel{\cong}{\longrightarrow} H_{n}(\Sp_{2n+6}(F))  \stackrel{\cong}{\longrightarrow} \cdots
$$
See \cite[Theorem 3.9]{Essert} and \cite[Theorem A]{SprehnWahl} which improve on \cite{MirzaiiII}.
\vspace{1ex}

As application of Theorems \ref{thmIntro:H3stab} and \ref{lem:Intro:apl} we generalise some results of Asok-Fasel \cite{AsokFasel:KODegree} from fields to local rings replacing the use of $\mathbb{A}^1$-homotopy theory with group homology computations.

We consider the $KO$-degree map \cite{AsokFasel:KODegree}
\begin{equation}
\label{eqn:KMW:SymbolMap}
K_n^{MW}(R) \to GW^{[n]}_n(R)
\end{equation}
from the Milnor-Witt $K$-groups of Hopkins-Morel \cite{Morel:book}, \cite{myEuler}, \cite{GilleEtAl}
to the higher Grothendieck-Witt-groups of \cite{Karoubi:batelle} in the form of \cite{myJPAA}.
The following is a combination of Theorems \ref{thm:GW22} and \ref{thm:GW33} in the text.

\begin{theorem}
\label{thm:Degree3iso}
Let $R$ be a local ring containing an infinite field of characteristic not $2$.
Then the $KO$-degree map (\ref{eqn:KMW:SymbolMap}) is an isomorphism for $n=2,3$.
\end{theorem}

The restriction to rings that contain an infinite field of characteristic not $2$ in Theorem \ref{thm:Degree3iso} comes from our use of \cite{GilleEtAl}.
This is probably unnecessary in light of Theorems \ref{thm:HurewiczGW22} and \ref{thm:HurewiczGW33}.
 \vspace{2ex}
 
 Finally, we obtain an interpretation of the indecomposable part of $K_3$ 
 in terms of orthogonal $K$-theory.
 Recall that when $\frac{1}{2}\in R$, the orthogonal $K$-group $KO_3(R)$ is the higher Grothendieck-Witt group $GW^{[4]}_3(R) = \pi_3BO(R)^+$ where $O(R)=\bigcup_{n\geq 0}O_{2n}(R)$ is the infinite orthogonal group of $R$, that is, the union of the groups of $R$-linear automorphisms preserving the standard hyperbolic quadratic form of rank $2n$.
 The following is Corollary \ref{cor:Kind} in the text.
  
  \begin{theorem}
  \label{thm:K3ind}
  	Let $R$ be a commutative local ring containing an infinite field of characteristic not $2$. Then
  	$$K_3^{ind}(R) \cong KO_3(R).$$
  \end{theorem}
  
Using   $\mathbb{A}^1$-homotopy theory, Theorems \ref{thm:Degree3iso} and \ref{thm:K3ind} were proven for infinite fields of characteristic not $2$ by Asok--Fasel in \cite{AsokFasel:KODegree}.
\vspace{2ex}

As is the case for most papers on homology stability of groups, we construct a highly connected complex on which our groups act, and we analyse the resulting spectral sequence. 
Our innovation is the use of the complex of non-degenerate unimodular sequences, a subcomplex of the complex of unimodular sequences used in \cite{SuslinNesterenko} and \cite{HutchinsonTao}, \cite{myEuler} to prove optimal homology stability for general linear and special linear groups.
This leads to the introduction of odd rank symplectic groups $\Sp_{2n-1}(R)$, and we prove a homology stability result in Theorem \ref{them:H3stability}
for the extended sequence of groups
$$\Sp_{-1}(R) = \Sp_0(R) \subset \Sp_1(R) \subset \Sp_2(R) \subset \dots \subset \Sp_n(R) \subset \Sp_{n+1}(R)\subset \dots $$
Finally, we use the localisation techniques introduced by the first author in \cite[\S2]{myEuler} and \cite[Appendix D]{myForm1}  to relate the homology of $\Sp_{2n+1}(R)$ to that of $\Sp_{2n}(R)$ in Proposition \ref{prop:HSpLocalization}.
\vspace{1ex}

For most of the paper, we fix a commutative local ring $R$ with infinite residue field and often suppress $R$ from the notation. 
For instance, $\Sp_n$, $U_q$, $\Skew_q^+$ will mean $\Sp_n(R)$, $U_q(R)$, $\Skew_q^+(R)$.
  
\section{The complex of non-degenerate unimodular sequences}
\label{sec:NonDegUniMSeq}

Throughout this paper, ring means commutative ring. Let $R$ be a ring. 
Let $R^*$ denote the group of units of $R$ under multiplication and $\GL_{n}(R)$  the
group of all invertible $n\times n$ matrices with entries in $R$.

Let $\psi_{2n}=\psi_2 \oplus \cdots \oplus \psi_2$ be the standard hyperbolic symplectic form of rank $2n$
$$\psi_{2n} = \left(\begin{smallmatrix}\psi_2 &&&\\ & \psi_2 && \\ && \ddots & \\ &&& \psi_2\end{smallmatrix}\right) = \bigoplus_1^n\psi_2,\hspace{3ex} \psi_2 = \left(\begin{smallmatrix}0 & 1 \\ -1 & 0 \end{smallmatrix}\right).$$

For a ring $R$, the {\em symplectic group} $\Sp_{2n}(R) \subset \GL_{2n}(R)$, is the subgroup 
$$\Sp_{2n}(R) =\{ A \in \GL_{2n}(R)|\ {^t\!A}\, \psi_{2n}\, A = \psi_{2n}\}$$
of $R$-linear automorphisms preserving the form $\psi_{2n}$ 
where $^t\!A$ denotes the transpose matrix of $A$.
We will always consider $\Sp_{2n}(R)$ as a subgroup of $\Sp_{2n+2}(R)$ via the embedding
\begin{equation}
\label{eqn:GSpIncls}
\Sp_{2n}(R) \subset \Sp_{2n+2}(R): A \mapsto \left(\begin{smallmatrix}1 & 0 & 0 \\ 0 & 1 & 0 \\ 0 & 0 & A\end{smallmatrix}\right).
\end{equation}

For $n\geq 0$, the {\em symplectic group of rank $2n+1$} is the subgroup
$$\Sp_{2n+1}(R) = \{A \in \Sp_{2n+2}(R)|\ Ae_1 = e_1\}$$
of $\Sp_{2n+2}(R)$ fixing the first standard basis vector $e_1$.
This is the group of matrices under multiplication
\begin{equation}
\label{eqn:GSpodd}
\left(\begin{smallmatrix}1& c & {^t\!u}\psi M  \\ 0 & 1 & 0 \\  0 & u & M \end{smallmatrix}\right)
\end{equation}
where $\psi = \psi_{2n}$, $M\in \Sp_{2n}(R)$, $u\in R^{2n}$, $c\in R$.
We let $\Sp_{-1}(R)=\{1\}$ be the trivial group. 
The inclusions (\ref{eqn:GSpIncls}) refine to the sequence of inclusions of groups
\begin{equation}
\label{eqn:SpInclusions}
\Sp_{-1}(R) = \Sp_0(R) \subset \Sp_1(R) \subset \Sp_2(R) \subset \dots \subset \Sp_n(R) \subset \Sp_{n+1}(R)\subset \dots
\end{equation}
where
\begin{equation}
\Sp_{2n}(R) \subset \Sp_{2n+1}(R): M \mapsto \left(\begin{smallmatrix}1 & 0 & 0 \\ 0 & 1 & 0 \\ 0& 0 & M \end{smallmatrix}\right), \hspace{3ex} \Sp_{2n-1} (R)\subset \Sp_{2n}(R): M \mapsto M.
\end{equation}
We shall denote the inclusions $\Sp_{r}(R) \subset \Sp_s(R)$ by $\eps^s_r$, or simply by $\eps$ if source and target group are understood, $r\leq s$.
Small rank symplectic groups are as follows
$$\Sp_{-1}(R)=\Sp_0(R) = \{1\},\hspace{2ex} \Sp_1(R) =\left \{\left(\begin{smallmatrix}1 & c \\ 0 & 1\end{smallmatrix}\right)|\  c\in R\right\},\hspace{2ex}\Sp_2(R) = \SL_2(R).$$
We will study homology stability for the sequence of groups (\ref{eqn:SpInclusions}).
\vspace{2ex}

A {\em space} over a ring $R$ is a projective $R$-module of finite rank.
A map of spaces is a map of $R$-modules between spaces.
A submodule $M\subset V$ of a space $V$ is called {\em subspace} if it is a direct factor.
A map of spaces $V \to W$ is called {\em split} if its image is a subspace.
An element $x\in V$ is called {\em unimodular} if it generates a subspace $Rx \subset V$.

A {\em bilinear space} $V$ over a commutative ring $R$ is a space over $R$ equipped with an $R$-bilinear form $V\times V \to R:(x,y)\mapsto \langle x,y\rangle$.
It is called {\em symplectic} if $\langle x,x\rangle =0$ for all $x\in V$.
Note that symplectic forms satisfy $\langle x,y\rangle =- \langle y,x\rangle$.
A symplectic space $V$ is called {\em non-degenerate} or {\em regular} if the adjoint map $V \to V^* = \Hom_R(V,R):x \mapsto \langle x,\phantom{y}\rangle$ is split with kernel of minimal possible rank, i.e.,
if $V$ has even rank then the kernel is required to be  $0$ (in which case the adjoint map is an isomorphism), and if $V$ has odd rank, then the kernel is required to have rank $1$.

From now on, let $R$ be a commutative local ring.
A basis $v=(v_1,...,v_q)$ of a space $V$ over $R$ of rank $q$ determines the dual basis $v^{\vee}=(v^{\vee}_1,...,v^{\vee}_q)$ of $V^*$ by the property $v^{\vee}_i(v_j) = \delta_{i,j}$.
If $V$ is equipped with a symplectic form then the adjoint map $V \to V^*$ in the basis $v$ and $v^{\vee}$ is the {\em Gram matrix} 
$$\Gamma(v) = (\langle v_i,v_j\rangle)_{i,j=1}^q$$
 of $v$.
In particular, the form is non-degenerate if and only if $\Gamma(v)$ is split with kernel of rank $0$ for $q$ even and of rank $1$ for $q$ odd.
A {\em symplectic basis} of a non-degenerate symplectic space $V$ of rank $2n$ is an ordered basis $v_1,v_2,....,v_{2n-1},v_{2n}$ of $V$ such that $\langle  v_{2r-1},v_{2r}\rangle =1$ and $\langle v_i,v_j\rangle =0$ for $r=1,...,n$, $1 \leq i<j\leq 2n$, $(i,j)\neq (2r-1,2r)$.
That is, a symplectic basis is an ordered basis whose Gram matrix is $\psi_{2n}$.
Recall \cite[Corollary 3.5]{MilnorHusemoller} that every non-degenerate symplectic space of even rank over a commutative local ring has a symplectic basis.

\begin{remark}
	\label{rmk:OddNonDeg}
	Let $v=(v_1,...,v_q)$ be a basis of a symplectic space $(V,\langle\phantom{x},\phantom{y}\rangle)$ of odd rank $q$.
	If $(v_1,...,v_{q-1})$ generates a non-degenerate subspace $W \subset V$ then $V$ is non-degenerate.
	This is because $V = W \perp W^{\perp}$ as $W$ is non-degenerate, and the orthogonal complement $W^{\perp} = \{x\in V|\ \langle x,y\rangle =0\ \forall y\in V\}$ of $W$ in $V$ is $W^{\perp}=Rx$ for some unimodular $x\in V$.
	Now, the Gram matrix in the basis $(v_1,...,v_{q-1},x)$ is $\left(\begin{smallmatrix}\Gamma(v_1,...,v_{q-1}) & 0 \\ 0 & 0\end{smallmatrix}\right)$ which is split of rank $q-1$.
\end{remark}


\begin{lemma}
	\label{lem:BasicsOddV}
	Let $R$ be a local ring and  ($R^{2n},\langle\phantom{x},\phantom{y}\rangle $)  a non-degenerate symplectic space.
	Let $V \subset R^{2n}$ be a non-degenerate subspace of rank $2r+1$. Then
	\begin{enumerate}
		\item
		$V\cap V^{\perp} = Rx$ for some unimodular $x\in R^{2n}$.
		\item
		$V$ contains a non-degenerate subspace of rank $2r$.
		\item
		\label{lem:BasicsOddV:item3}
		If $V_0 \subset V$ is a non-degenerate subspace of rank $2r$, then 
		$V = V_0 \perp Rx$ with $Rx=V\cap V^{\perp}$, $x$ unimodular, and there is $y\in R^{2n}$ such that $y \in V_0^{\perp}$ and $\langle x,y\rangle =1$.
	\end{enumerate}
\end{lemma}

\begin{proof}
	\begin{enumerate}
		\item
The $R$-linear map $\rho: V \to V^*: v \mapsto \langle v,\ \ \rangle$ has kernel $\ker(\rho)=V\cap V^{\perp}$ which is a rank $1$ subspace as $V$ is non-degenerate of odd rank.
		\item
		The form on $V$ induces a unique non-degenerate form on $V/(V\cap V^{\perp})$ such that the quotient map $V \to V/(V\cap V^{\perp})$ preserves forms.
		Any $V_0 \subset V$ mapping isomorphically onto $V/(V\cap V^{\perp})$ is a non-degenerate subspace of rank $2r$.
		\item
		Any non-degenerate $V_0\subset V$ of rank $2r$ induces a map $V_0 \to V/(V\cap V^{\perp})$ preserving forms. Since $V_0$ is non-degenerate and both have the same rank, this map is an isometry, and $V = V_0 \perp (V\cap V^{\perp})$.
		We have $V\cap V^{\perp}=Rx$ for some unimodular $x$.
Note that $x\in V_0^{\perp}$.
Since the symplectic form on $R^{2n}$ is non-degenerate
its restriction to $V_0^{\perp}$ is non-degenerate, too.
In particular, there is $y\in V_0^{\perp}$ such that $\langle x,y\rangle=1$.
	\end{enumerate}
\end{proof}

The following is well-known when the subspaces have even rank.

\begin{corollary}[Witt's Theorem]
	\label{cor:WittThm}
	Let $R$ be a local ring and  ($R^{2n},<,>$) a non-degenerate symplectic space.
	Let $V,W \subset R^{2n}$ be non-degenerate subspaces. 
	If $V$ and $W$ have the same rank  then there is an isometry $V \cong W$.
	Moreover, any isometry $g_0:V \cong W$ extends to an isometry $g:R^{2n} \to R^{2n}$ such that $g_{|V}=g_0$.
\end{corollary}

\begin{proof}
If $V$ and $W$ have even rank, they and their orthogonals $V^{\perp}$ and $W^{\perp}$ all have symplectic basis.
Hence there are isometries $V \cong W$ and $V^{\perp} \cong W^{\perp}$.
Any isometry $g_0: V \cong W$ extends to an isometry $g=g_0\perp g_1$ of $R^{2n}=V\perp V^{\perp} = W \perp W^{\perp}$ by choosing an isometry $g_1:V^{\perp} \cong W^{\perp}$.
	
Now assume that $V$ and $W$ have odd rank $2r+1$.
By Lemma \ref{lem:BasicsOddV}$(1)$, we have $V \cap V^{\perp} = Rx$, $W \cap W^{\perp} = Ry$ for 
some unimodular $x,y\in R^{2n}$. By Lemma \ref{lem:BasicsOddV}$(2)$, we can choose non-degenerate subspaces
$V_0 \subset V$ and $W_0\subset W$ of rank $2r$.
Note that the linear map sending a symplectic basis of $V_0$ to a symplectic basis of $W_0$ and $x$ to $y$ is an isometry $V \cong W$.
Now, let $g_0:V \cong W$ be any isometry.
	By Lemma \ref{lem:BasicsOddV}$(3)$, we can choose $x',y'\in R^{2n}$ such that $x'\in V_0^{\perp}$ and $y'\in W_0^{\perp}$ and such that $\langle x,x'\rangle = \langle y,y'\rangle =1$.
	Then the isometry $g_0:V\cong W$ extends to an isometry $V\perp Rx' \cong W \perp Ry'$ by sending $x'$ to $y'$.
	Since the latter two spaces are non-degenerate of even rank, the last isometry extends to an isometry $g$ of $R^{2n}$.
\end{proof}

Let $q\geq 0$ be an integer.
A {\em unimodular sequence} of length $q$ in $R^{2n}$ is a
sequence $(v_1,...,v_q)$ of $q$ vectors $v_1,...,v_q$ in $R^{2n}$ such that each subsequence of length $r \leq \min(q,2n)$ generates a subspace of rank $r$.
Unimodular sequences were used in \cite{SuslinNesterenko} and \cite{HutchinsonTao}, \cite{myEuler} to prove optimal homology stability for general linear and special linear groups.
 
From now on consider $R^{2n}$ equipped with the standard symplectic form where the standard basis $e_1,...,e_{2n}$ has Gram matrix $\Gamma(e_1,...,e_{2n}) = \psi_{2n}$.
A unimodular sequence $(v_1,...,v_q)$ of length $q$ in $R^{2n}$  is called {\em non-degenerate} if
any subsequence of length $r \leq \min(q,2n)$ is a basis of a non-degenerate subspace of the symplectic space $R^{2n}$.
For an integer $q\geq 0$, let 
$$U_q(R^{2n}) = \{v =(v_1,...,v_q)|\ v\ \text{non-degenerate unimodular in } R^{2n}\}$$
 be the set of non-degenerate unimodular sequences 
 of length $q$ in $R^{2n}$.
The set $U_0(R^{2n})$ is the singleton set consisting of the empty sequence, and the set $U_q(R^0)$ is the singleton set with unique element the sequence $(0,0,...,0)$ of length $q$.
For a $q$-tuple $v=(v_1,...,v_q)$ of vectors $v_i\in R^{2n}$ and an ordered subset $I=\{i_1 < \ldots <i_r\} \subset \{1,...,q\}$, we write $v_{I}=(v_{i_1},\ldots,v_{i_r})$.
In view of Remark \ref{rmk:OddNonDeg}, a sequence $v =(v_1,...,v_q)$ of vectors in $R^{2n}$ is in $U_q(R^{2n})$ if and only if it is unimodular and for any subset $I\subset \{1,...,q\}$ of even cardinality $\leq \min(q,2n)$, the Gram matrix $\Gamma(v_I)$ is invertible.

We define a chain complex $(C_*(R^{2n}),d)$ where
$C_q(R^{2n})=\Z[U_q(R^{2n})]$ is the free abelian group generated by the set $U_q(R^{2n})$ and the $\Z$-linear differential is defined on the basis by
\begin{equation}
\label{CompxDiff}
d_q(v_1,...,v_q) = \sum_{i=1}^q(-1)^{i+1}(v_1,...,\hat{v}_i,...,v_q).
\end{equation}

Let $v =(v_1,...,v_q) \in U_q(R^{2n})$. 
A vector $x\in R^{2n}$ is said to be {\em in good position with respect to $v$} if $(v,x)\in U_{q+1}(R^{2n})$.

\begin{remark}
\label{rmk:goodPisitionk}
	Let $(R, m)$ be a local ring.
	If $k=R/m$ denotes the residue field of $R$, then $x$ is in good position with respect to $v$ if and only if its class $\bar{x}\in k^{2n}$ modulo $m$ is in good position with respect to $\bar{v}=(\bar{v}_1,...,\bar{v}_q) \in U_q(k^{2n})$ over $k$.
\end{remark}

In this paper, a matrix $A=(a_{ij}) \in M_{q}(R)$ is called {\em skew symmetric} if $a_{ij}=-a_{ji}$ and $a_{ii}=0$ for all $1\leq i,j \leq q$.
For $q\geq 0$, let $\Skew_q(R) \subset M_q(R)$ be the set of skew symmetric $q\times q$-matrices with entries in $R$. 
For a skew symmetric matrix $A=(a_{ij}) \in M_{2n}(R)$, 
we denote by
$\Pf(A)$ its {\em Pfaffian}.
It can be recursively defined by the formula
\begin{equation}\label{eqn:pafdef}
\Pf(A) = \sum_{i=1}^{2n-1}(-1)^{i+1}a_{i,2n}\Pf(A_{\widehat {i,2n}})
\end{equation}
where $A_{\widehat {i j}}$ denotes the matrix $A$
with both the $i$-th  and $j$-th rows and columns removed,
and the Pfaffian of the unique $0\times 0$ matrix is $1$. 
For instance,
$$\Pf\begin{pmatrix}0 & a\\ -a & 0 \end{pmatrix}=a,\hspace{4ex}
\Pf\left(\begin{smallmatrix} 0 & a & b & c \\ -a & 0 & d & e \\ -b & -d & 0 & f\\ -c & -e & -f & 0\end{smallmatrix}\right) = af - be +cd.$$
The Pfaffian satisfies $\Pf(A)^2=\det(A)$, $\Pf({^t\!U}AU) = \det(U)\Pf(A)$, $\Pf(cA) = c^n\Pf(A)$ and $\Pf(\psi_{2n})=1$ for all $A\in \Skew_{2n}(R)$ and $c\in R$.

\begin{lemma}
	\label{lem:goodPosition}
	Let $k$ be a  field, $q\geq 0$, $n\geq 1$ integers, and let $v =(v_1,...,v_q) \in U_q(k^{2n})$ be a non-degenerate unimodular sequence of length $q$ in $k^{2n}$ equipped with the standard symplectic form $\langle\phantom{x},\phantom{y}\rangle$. 
	Then there is a finite set of subspaces $V_1,...,V_s \subset k^{2n}$ of rank $<2n$ such that every $x \in k^{2n} - \bigcup_{j=1}^sV_j$ is in good position with respect to $v$.
\end{lemma}

\begin{proof}
Let $I\subset \{1,...,q\}$ be a subset of cardinality $<2n$.
If $I$ has even cardinality, let $V_I \subset k^{2n}$  be the subspace generated by $v_I$.
Then for any $x \in k^{2n}\setminus V_I$, the vectors $v_i$, $i\in I$ and $x$ generate a non-degenerate subspace of $k^{2n}$, and $\dim V_I <2n$.
If $I$ has odd cardinality, consider the $k$-linear map $k^{2n} \to k:$
$$x \mapsto \Pf(\Gamma(v_I,x)) = \langle u,x\rangle$$
where $u =  \sum_{i\in I}\eps_i  \Pf(\Gamma(v_{I-\{i\}}))v_i$ and $\eps_i$ is $1$ or $-1$;
see (\ref{eqn:pafdef}).
Let $V_I=u^{\perp}$ be the kernel of that map.
Since each $ \Pf(\Gamma(v_{I-\{i\}}))$ is a unit ($|I|$ being odd), the vector $u$ is unimodular and $\dim(u^{\perp}) < 2n$.
Then for any $x \in k^{2n}\setminus V_I$, the vectors $v_i$, $i\in I$ and $x$ generate a non-degenerate subspace of $k^{2n}$, and $\dim V_I <2n$.
Now any $x\in k^{2n} \setminus \bigcup_{I \subset \{1,...,q\},\ |I|<2n} V_I$ is in good position with respect to $v$.
\end{proof}

\begin{corollary}
	\label{cor:C_*acyclic}
	Let $(R,m)$ be a local ring with infinite residue field.
	Then the chain complex $C_*(R^{2n})$ is acyclic, that is, for all $q\geq 0$, we have
	$$H_q(C_*(R^{2n})) = 0.$$
\end{corollary}

\begin{proof}
	Let $\xi=\sum_{i=1}^rn_iv_i\in C_q(R^{2n})$ be a cycle where $n_i\in \Z$ and $v_i\in U_q(R^{2n})$.
	Since $R/m$ is infinite, there is $x\in R^{2n}$ which is in good position with respect to all $v_i$, $i=1,...,r$, by Lemma \ref{lem:goodPosition} and Remark \ref{rmk:goodPisitionk}.
	Then $(x,\xi) \in C_{q+1}$ and $d (x,\xi) = \xi - (x,d\xi) = \xi$, that is, $\xi$ is a boundary.
\end{proof}

Let $\Skew_q^+(R)\subset \Skew_q(R)$
 be the subset of those skew symmetric matrices $A$ such that for all $\emptyset \neq I \subset \{1,...,q\}$ of even cardinality, the submatrix $A_I$ with columns and rows in $I$ is invertible.
For $q=0,...,2n+1$, the map 
$$\Gamma: U_q(R^{2n})/\Sp_{2n}(R) \rightarrow \Skew_q(R): 
 v=(v_1,...,v_q) \mapsto \Gamma(v) =  (\langle v_i,v_j\rangle)_{i,j}$$
sending a non-degenerate unimodular sequence to its Gram matrix  has image in $\Skew^+_q(R)$.
\vspace{1ex}

For integers $q \geq j \geq 1$, a sequence $v$ of length $q$ and $A \in \Skew_q$ a skew-symmetric matrix, denote by $v_{\hat{j}}$ the sequence of length $q-1$ obtained from $v$ by removing the $j$-th entry, and denote by $A_{\widehat{j}}\in \Skew_{q-1}$ the skew-symmetric matrix obtained from $A$ by removing the $j$-th row and column.

\begin{construction}
\label{const:GammaSections}
For integers $0 \leq i \leq 2n+1$ there are unique sections
$$V_{i,2n} : \Skew_{i}^+(R) \to U_i(R^{2n})$$
of $\Gamma$ such that
\vspace{1ex}

\begin{enumerate}
\item
\label{const:GammaSections:item1}
$V_{0,2n}$ is the empty sequence, and $V_{1,0}=(0)$.
\vspace{1ex}

\item
\label{const:GammaSections:item2}
$V_{i,2n}(A) = V_{i,2m}(A)$ for $i\leq 2n \leq 2m$ under the standard embedding $R^{2n} \subset R^{2m}: e_i \mapsto e_i$.
\vspace{1ex}

\item
\label{const:GammaSections:item3}
$V_{i+1,2n}(A)_{\widehat{i+1}} = V_{i,2n}(A_{\widehat{i+1}})$,  $A\in \Skew_{i+1}^+(R)$, $0\leq i \leq 2n$.
\vspace{1ex}

\item
\label{const:GammaSections:item4}
For $A \in \Skew^+_{2n+1}(R)$ we have
$$V_{2n+1,2n}(A) = (v_1,...,v_{2n},w_{2n+1}),\hspace{3ex} V_{2n+1,2n+2}(A) = (v_1,...,v_{2n},v_{2n+1})$$
where 
$(v_1,...,v_{2n}) = V_{2n,2n}(A_{\widehat{2n+1}})$, 
$w_{2n+1} \in R^{2n}$ is the unique element such that $\langle v_i,w_{2n+1}\rangle = A_{i,2n+1}$ for all $0 \leq i \leq 2n$, 
and 
$$v_{2n+1}=w_{2n+1}+ e_{2n+1}.$$

\item
\label{const:GammaSections:item5}
For $A \in \Skew^+_{2n+2}(R)$, 
write $V_{2n+1,2n+2}(A_{\widehat{2n+2}}) = (v_1,...,v_{2n},v_{2n+1})$ and 
$V_{2n+1,2n}(A_{\widehat{2n+1}}) = (v_1,...,v_{2n},w_{2n+2})$. 
Then $V_{2n+2,2n+2}(A) = (v_1,...,v_{2n+2})$ with
$$
v_{2n+2} = w_{2n+2}
+  \left(A_{2n+1,2n+2}-\langle v_{2n+1},w_{2n+2}\rangle\right) e_{2n+2}.
$$
\end{enumerate}
From the recursive nature of the construction satisfying (\ref{const:GammaSections:item1}) - (\ref{const:GammaSections:item5}) it is clear that the sections exist and are unique. 
\end{construction}

\begin{lemma}
	\label{lem:GammaBij}
For $0 \leq q \leq 2n+1$ the following map is bijective
	$$\Gamma:U_q(R^{2n})/\Sp_{2n}(R) \to \Skew_q^+(R).$$
\end{lemma}

\begin{proof}
Surjectivity for $0 \leq q \leq 2n+1$ follows from the existence of the sections in Construction \ref{const:GammaSections}.

For injectivity, let $v=(v_1,...,v_q)$ and $w=(w_1,...,w_q)$ be in $U_q(R^{2n})$ and assume that $\Gamma(v)=\Gamma(w)$.
	First assume $q\leq 2n$.
	Then $v$ and $w$ span non-degenerate subspaces $V$ and $W$ of $R^{2n}$.
	Since $\Gamma(v)=\Gamma(w)$, the linear map sending $v_i$ to $w_i$ is an isometry $V\cong W$.
	By Witt's theorem (Corollary \ref{cor:WittThm}), this isometry extends to an isometry of $R^{2n}$.
	In particular, $[v] = [w]\in U_q(R^{2n})/\Sp_{2n}(R)$.
	
	Now assume $q = 2n+1$. There is a unique $g\in \Sp_{2n}(R)$ such that $(v_1,...,v_{2n}) = g(w_1,...,w_{2n})$. 
	So, we can assume $(v_1,...,v_{2n}) = (w_1,...,w_{2n})$.
	Now the bijectivity of the map
	\begin{equation}
	\label{eqn:GammaSurj}
	R^{2n} \stackrel{\cong}{\longrightarrow} R^{2n}: v\mapsto (\langle v_i,v\rangle)_{i=1}^{2n}.
	\end{equation}
shows that we also have $v_{2n+1}=w_{2n+1}$. 
\end{proof}

For $i=1,...,q$, we  define the maps 
$$\Skew_q^+(R) \to \Skew_{q-1}^+(R): A \mapsto A^{\wedge}_i$$ omitting the $i$-th row and column.
We make the graded abelian group $\Z[\Skew_*^+(R)]$ into a chain complex with the differentials
\begin{equation}\label{SkewDiff}
d_q:\Z[\Skew_q^+(R)] \to \Z[\Skew_{q-1}^+](R): [A] \mapsto \sum_{i=1}^q (-1)^{i+1}[A^{\wedge}_i].
\end{equation}
It is easy to check that $d_qd_{q+1}=0.$

\begin{lemma}
	\label{lem:SkewExact}
	Let $R$ be a local ring with infinite residue field.
	Then the chain  complex $(\Z[\Skew_*^+(R)],d_*)$ is acyclic.
	 That is, for all $p\geq 0$ we have
	$$H_p(\Z[\Skew_*^+(R)])=0.$$
\end{lemma}

For $A\in \Skew^+_q(R)$ and $v\in R^q$, denote by $A\ast v$ the matrix
$$A\ast v = \left(\begin{smallmatrix}A & v \\ -{^t\!v} & 0 \end{smallmatrix}\right) \in \Skew_{q+1}(R).$$
Note that $A\ast v\in \Skew^+_{q+1}(R)$ if and only if for all $I = \{i_1< \cdots <i_r\} \subset \{1,...,q\} $ of odd cardinality $r$, the Pfaffian $\Pf(A_I\ast v_I)$ of $A_I\ast v_I$ is a unit in $R$, or equivalently non zero in the residue field $k$ of $R$.

\begin{proof}[Proof of Lemma \ref{lem:SkewExact}]
Let $\xi = \sum_{j=1}^mn_j[A_j]\in \Z[\Skew_q^+(R)]$.
We need to show that if $\xi$ is a cycle then it is a boundary.
Since $\Skew_1^+(R)=\Skew_0^+(R)$ is the singleton set, this is clear for $q=0$, and we can assume $q\geq 1$.

Let $k$ be the residue field of $R$ and let $I = \{i_1< \cdots <i_r\} \subset \{1,...,q\}$ have odd cardinality.
The recursive formula for the Pfaffian (\ref{eqn:pafdef}) shows that the map $k^{q} \to k:$
$$\bar{v} \mapsto \Pf(A_I\ast \bar{v}_I) = (\pm\Pf(A_{I-i_1}),...,\pm\Pf(A_{I-i_r}))\cdot \bar{v}_I$$
is $k$-linear. 
Since $A\in \Skew_q^+(R)$, all entries $\pm\Pf(A_{I-i_s})$ are units and the $k$-linear map is non-zero.
Thus, its kernel $V_{A,I}$ has codimension $1$ in $k^{q}$.

Since the residue field $k$ of $R$ is infinite, the set
$$k^{q}  \setminus \hspace{2ex} \bigcup_{j=1,..., m, I \subset \{1,...,q\},\ |I|\ \text{odd}}V_{A_j,I}$$
is non-empty. 
In particular, there is $v\in R^{q}$ which is mapped into that set under the map $R \to k$.
By the discussion above, $A_j\ast v\in \Skew^+_{q+1}(R)$ for all $j=1,...,m$.
Then $\xi\ast v =  \sum_{j=1}^mn_j[A_j\ast v] \in \Z[\Skew_{q+1}^+(R)]$ and
$d(\xi\ast v) = (-1)^q \xi + (d\xi)\ast v$.
In particular, if $\xi$ is a cycle then it is a boundary.
\end{proof}

\begin{corollary}
	Let $R$ be a local ring with infinite residue field.
	Then the complex $\Z[U_*(R^{2n})/\Sp_{2n}(R)]$ is acyclic in degrees $\leq 2n$, that is,
	$$H_p(\Z[U_*(R^{2n})/\Sp_{2n}(R)]) = 0,\hspace{3ex} 0 \leq p \leq 2n.$$
\end{corollary}

\begin{proof}
	Combine Lemmas \ref{lem:GammaBij} and \ref{lem:SkewExact}.
\end{proof}

\section{The Homology spectral sequence and its $E^1$-page}

For a group $G$, let $\Z[G]$ denote its integral group ring. 
Most of our computations are concerned with the homology \cite{Br82}
$$H_*(G;M)=Tor_*^{\Z[G]}(\Z,M) = H_*\left( \Z[EG]\otimes_G M\right)$$
 of $G$ with coefficient in a (bounded below) complex of left $G$-modules $M$.
Here $\Z[EG]$ is the complex of right $G$-modules which in degree $n$ is the free $\Z$-module on the right $G$-set $E_nG =G^{n+1}$ and differential $d_n$ defined on basis elements by 
 $d_n(g_0,...,g_n) = \sum_{i=0}^n(-1)^i(g_0,...,\hat{g}_i,..,g_n)$.
The complex $\Z[EG]$ is a resolution of the trivial $G$-module $\Z$ by the free right $\Z[G]$-modules $\Z[E_nG]$ with basis $B_nG=G^n$.
 
\begin{example}
If $M=\Z[U]$ is the free group on a left $G$-set $U$, the low degree homology groups $H_i(G,\Z[U])$, $i=0,1$, are the homology groups of the complex
$$\Z[G^2\times U]  \stackrel{d_2}{\longrightarrow} \Z[G\times U]  \stackrel{d_1}{\longrightarrow} \Z[U]  \stackrel{d_0}{\longrightarrow} 0$$
where $d_2[g,h,u] = [h,u]-[gh,u]+[g,hu]$ and $d_1[g,u] = [u] - [gu]$, $g,h\in G$, $u\in U$.
In particular, $H_0(G,\Z[U]) = \Z[G\backslash U]$ and $H_1(G,\Z[U])$ is the group of elements $\sum_{i=1}^nm_i[g_i,u_i] \in \Z[G\times U]$ satisfying $\sum_{i=1}^nm_i[g_iu_i] = \sum_{i=1}^nm_i[u_i]  \in \Z[U]$ modulo the relation $[gh,u]=[h,u]+[g,hu]$.
\end{example}

\begin{example}
\label{ex:H1G}
Continuing the previous example, if $U = \ast$ is the one-element set, then we have the isomorphism
$$H_1(G,\Z) \stackrel{\cong}{\longrightarrow} G^{ab}: \sum_{i=1}^nm_i[g_i] \mapsto \prod_{i=1}^ng_i^{m_i}.$$
\end{example}

Consider the category $\mathcal{C}$ whose objects are the pairs $(G,M)$, where $G$ is a group
and $M$ is a bounded below complex of left $G$-modules. 
The arrows $(G,M) \to (G',M')$ in $\mathcal{C}$ are pairs $(\alpha, f)$ where $\alpha: G \rightarrow G'$ is a group homomorphism and
$f: M\rightarrow M'$ is a chain map such that $f(a\cdot x) = \alpha(a)\cdot f(x)$ for $a\in G$ and $x\in M$.
Composition in $\C$ is composition of the $\alpha$'s and $f$'s. 
An arrow  $(\alpha, f):(G,M) \rightarrow (G',M')$ defines a map on homology 
$$(\alpha,f)_*: H_*(G;M)\rightarrow  H_*(G';M')$$ 
induced by the chain map 
$$\Z[EG]\otimes_GM \to \Z[EG']\otimes_{G'}M': (a_0,...,a_n)\otimes x \mapsto (\alpha(a_0),...,\alpha(a_n)) \otimes f(x).$$
Let $(\alpha_0,f_0),(\alpha_1, f_1):(G,M) \rightarrow (G',M')$ be two arrows in $\C$. 
Assume for simplicity that $M$ and $M'$ are $G$ and $G'$-modules, respectively.
If there is an element
$h\in G'$ such that $\alpha_1(a)=h\alpha_0(a)h^{-1}$ and $f_1(x)=hf_0(x)$ for all $a\in G$ and
 $x\in M$, then the induced maps on homology agree:
 \begin{equation}
 \label{eqn:HinducedEqual}
 (\alpha_0,f_0)_*=(\alpha_1,f_1)_*: H_*(G;M)\rightarrow H_*(G';M').
 \end{equation}

 For a complex $M$ of left $G$-modules with $M_i=0$ for $i<0$, the stupid filtration $M_{\leq 0} \subset M_{\leq 1} \subset M_{\leq 2} \subset \dots \subset M$ of $M$ defined by 
\begin{equation}
\label{eqn:TruncCx}
 (M_{\leq r})_i = \left\{\begin{array}{cl} M_i & i\leq r\\ 0 & i>r \end{array}\right.
 \end{equation}
 has quotients $M_{\leq q}/M_{\leq q-1}$ the $G$-module $M_q$ placed in homological degree $q$. 
This defines the spectral sequence
 \begin{equation}
 \label{eqn:GpHomSpSeq}
 E^1_{p,q}=H_p(G;M_q) \Rightarrow H_{p+q}(G;M)
 \end{equation}
 with differential $d^r$ of bidegree $(r-1,-r)$.
The spectral sequence is functorial for maps in $\C$.
\vspace{1ex}

Let us recall the chain complex $(C_*(R^{2n}), d)$ of left $\Sp_{2n}(R)$-modules from Section \ref{sec:NonDegUniMSeq} and its truncation (\ref{eqn:TruncCx}), the subcomplex $C_{\leq 2n+1}(R^{2n})$, which is the free abelian group $C_q(R^{2n})=\Z[U_q(R^{2n})]$  on the set of non-degenerate unimodular sequences of length $q$ in $R^{2n}$ for $0 \leq q \leq 2n+1$ and which is $0$ otherwise. 
The differential was defined in (\ref{CompxDiff}).

\begin{lemma}
\label{lem:HSpC}
Let $R$ be a local ring with infinite residue field.
Then for all integers $n,r$ with $0 \leq r \leq 2n$ we have
$$H_r(\Sp_{2n}(R);C_{\leq 2n+1}(R^{2n}))= 0.$$
\end{lemma}

\begin{proof}
For a group $G$ and a bounded below complex of left $G$-modules $M$, 
the stupid filtration $\Z[EG]_{\leq 0} \subset \Z[EG]_{\leq 1} \subset \Z[EG]_{\leq 2} \subset ... $ of $\Z[EG]$ defined by 
$$(\Z[EG]_{\leq q})_r = \left\{\begin{array}{ll} \Z[E_rG], &  r\leq q\\ 0, & r>q \end{array}\right.$$
induces a spectral sequence
$$E^1_{p,q} = H_p(\Z[E_qG]\otimes_G M) = \Z[B_qG]\otimes H_p(M) \Rightarrow H_{p+q}(G;M)$$
where $B_qG=G^q$.
For $G=\Sp_{2n}(R)$ and $M = C_{\leq 2n+1}(R^{2n})$, we have $E^1_{p,q}=0$ for $p\leq 2n$ in view of Corollary \ref{cor:C_*acyclic}.
The spectral sequence then implies that $H_r(\Sp_{2n}(R);C_{\leq 2n+1} (R^{2n}))=0$ for all $r\in \Z$ with $r \leq 2n$.
 \end{proof}

For $G=\Sp_{2n}(R)$ and $M=C_{\leq 2n+1}(R^{2n})$, the spectral sequence (\ref{eqn:GpHomSpSeq}) has the form
 \begin{equation}
 \label{SpecSeq}
E^1_{p,q}({2n})   \Rightarrow H_{p+q}(\Sp_{2n}(R),C_{\leq 2n+1}(R^{2n}))
\end{equation}
 where the differential $d^r$ is of bidegree $(r-1,-r)$ and
 $$E^1_{p,q}({2n})  = 
 \left\{
 \renewcommand\arraystretch{1.5}
 \begin{array}{rl}
 H_p(\Sp_{2n}(R), \Z[U_q(R^{2n})]) & q \leq 2n+1\\
 0 & q>2n+1.
 \end{array}\right.
 $$
 The spectral sequence converges to $0$ for $p+q\leq 2n$ in view of Lemma \ref{lem:HSpC}.
We need to determine explicitly the $d^1$-differentials
$d^1_{p,q}: E^1_{p,q}({2n})\rightarrow E^1_{p,q-1}({2n})$ which, for $q\leq 2n+1$, are 
the maps
$$d^1_{p,q} = (1,d)_*: H_p(\Sp_{2n}(R),\Z[U_q(R^{2n})]) \to H_p(\Sp_{2n}(R),\Z[U_{q-1}(R^{2n})])$$
where $d:\Z[U_{q}(R^{2n})] \to \Z[U_{q-1}(R^{2n})]$ is the differential of the complex $C_*(R^{2n})$.
In order to do so, we recall Shapiro's Lemma.
Let $G$ be a group acting on a set $S$ from the left.
Shapiro's Lemma gives an isomorphism
$$\bigoplus (i_x, x)_*: \bigoplus_{[x]\in G\backslash S} H_*(G_x;\Z) \stackrel{\cong}{\longrightarrow} H_*(G;\Z[S])$$
of homology groups
where the direct sum is over a set of representatives $x \in S$ of equivalences classes $[x] \in G\backslash S$, the group $G_x = \{a\in G|\ ax=x\}$ is the stabiliser of $G$ at $x\in S$, the homomorphism $i_x:G_x \subset G$ is the inclusion, and $x$ also denotes the homomorphism of abelian groups $\Z \to \Z[S]: 1 \mapsto x$.

We will apply Shapiro's Lemma to $G=\Sp_{2n}(R)$ and $S=U_q(R^{2n})$.
To ease notation we will write $\St_{2n}(v)$ for the stabiliser of $\Sp_{2n}(R)$ at a sequence $v = (v_1,...,v_q)$ of vectors $v_i$ in $R^{2n}$.
For the sections of $\Gamma$ defined in Construction \ref{const:GammaSections} we have $\St_{2n}(V_{q,2n}(A)) = \St_{2n}(e_1,...,e_q) = \Sp_{2n-q}(R)$ for $q=0,...,2n$ since $V_{q,2n}(A) = (v_{1},...,v_{q})$ and $(e_{1},...,e_{q})$ span the same subspace of $R^{2n}$.\
Moreover, $\St_{2n}(V_{2n+1,2n}(A)) = \St_{2n}(e_1,...,e_{2n}) = \{1\}= \Sp_{-1}(R)$ is the trivial group.

Recall from Lemma \ref{lem:GammaBij} that the Gram matrix defines a bijection of sets
$$\Gamma: \Sp_{2n}(R)\backslash U_q(R^{2n}) \stackrel{\cong}{\longrightarrow} \Skew_q^+(R): [v]\mapsto \Gamma(v),\hspace{2ex}0\leq q \leq 2n+1.$$

In the following we shall denote the inclusions $\Sp_{r}(R) \subset \Sp_s(R)$ generically by $\eps$ for $r\leq s$.
By Shapiro's lemma, the following map is an isomorphism
\begin{equation}
\label{eqn:E1ShapiroIdentn}
 H_*(Sp_{2n-q}; \Z)\otimes_{\Z}\Z[\Skew_q^+]  \stackrel{\cong}{\longrightarrow} H_*(Sp_{2n}; \Z[U_q(R^{2n})]), \hspace{2ex} 0 \leq q \leq 2n+1
\end{equation}
where $\alpha \otimes A$ is sent to $(\eps,V_{q,2n}(A))_*(\alpha)$.

\begin{lemma}
\label{lem:commdiag}
	For $0 \leq q \leq 2n$ the following diagram is commutative
\begin{equation}\label{eqn:commdaig}
\xymatrix{
	H_*(\Sp_{2n-q-1}; \Z) \otimes_{\Z} \Z[\Skew_{q+1}^+] \ar[rr]^{\hspace{5ex}\text{(\ref{eqn:E1ShapiroIdentn})}}_{\hspace{5ex}\cong}
	\ar@{->}[d]_{\eps_*\otimes d}
	&& H_*(\Sp_{2n}; \Z[U_{q+1}(R^{2n})])
	\ar@{->}[d]^{(1,d)_*}
	\\
	H_*(\Sp_{2n-q}; \Z)\otimes_{\Z} \Z[\Skew_q^+]   \ar[rr]^{\hspace{7ex}\text{(\ref{eqn:E1ShapiroIdentn})}}_{\hspace{7ex}\cong}
	&&H_*(\Sp_{2n}; \Z[U_q(R^{2n})]).
}
\end{equation}

\end{lemma}

\begin{proof}
Recall that $d = \sum_{i=1}^{q+1} (-1)^{i+1}d_i$ where $d_i$ omits the $i$-th entry.
Write $V_q$ for the section $V_{q,2n}$ defined in Construction \ref{const:GammaSections}.
We will show that for all $A \in \Skew_{q+1}^+(R)$ and all $1 \leq i \leq q+1$, the two maps
$$(\eps,d_i V_{q+1}(A))_*, \ (\eps, V_q(d_iA))_*:  H_*(Sp_{2n-q-1}; \Z) \longrightarrow H_*(\Sp_{2n}; \Z[U_q(R^{2n})])$$
are equal.
For $A \in \Skew_{q+1}^+(R)$ we will construct $B \in \Sp_{2n}(R)$ such that 
$$B(d_iV_{q+1}(A)) = V_q(d_iA)\text{ and }BC=CB\text{ for all } C\in \Sp_{2n-q-1}(R)\subset \Sp_{2n}(R).$$
The existence of such $B$ implies $(\eps,d_i V_{q+1}(A))_*=(\eps, V_q(d_iA))_*$; see (\ref{eqn:HinducedEqual}).

Assume $q$ is odd.
All entries of $V_q(d_iA)$ and $V_{q+1}(A)$, and hence of $d_iV_{q+1}(A)$, are vectors in the non-degenerate subspace $R^{q+1}\subset R^{2n}$ of even rank $q+1$.
The unimodular sequences $d_iV_{q+1}(A)$ and $V_q(d_iA)$ are basis of non-degenerate subspaces $W_1$ and $W_2$ of $R^{q+1}$ of rank $q$.
Since both sequences have the same Gram matrix, the linear isomorphism $W_1 \cong W_2$ sending $d_iV_{q+1}(A)$ to $V_q(d_iA)$ is an isometry.
By Corollary \ref{cor:WittThm} this isometry extends to an isometry $B:R^{q+1} \to R^{q+1}$ which we extend to all of $R^{2n}$ 
by the requirement $Be_i=e_i$ for $i=q+2,...,2n$.
Then $B\in \Sp_{2n}(R)$ satisfies $B(d_iV_{q+1}(A)) = V_q(d_iA)$ and commutes with 
any $C\in \Sp_{2n-q-1} (R)\subset \Sp_{2n}(R)$. 

Assume $q$ even.
If $q=0$ then $dV_{q+1}(A)=V_{q}(dA)$ is the empty sequence and we are done. 
If $q=2n$, then $d_iV_{q+1}(A)$ and $V_{q}(d_iA)$ are both basis of $R^{2n}$.
Since $\Gamma(dV_{q+1}(A))=\Gamma(V_{q}(dA))$, the linear isomorphism $B:R^{2n} \to R^{2n}$ sending 
$d_iV_{q+1}(A)$ to $V_{q}(d_iA)$ is an isometry which commutes with every element of $\Sp_{-1}(R)=\{1\}$.

Assume $q$ is even, $2 \leq q\leq 2n-2$ and $i=q+1$.
Then $d_iV_{q+1}(A)=V_{q}(d_iA)$ and we can choose $B$ to be the identity matrix.

Finally, assume $q=2r$ is even, $2 \leq q\leq 2n-2$ and $1 \leq i \leq q$.
Recall that $A \in \Skew_{2r+1}^+$. 
Writing $V_{q+1}(A) = (v_1, v_{2}, ..., v_{2r+1})$, we have $d_iV_{q+1}(A)= (v_{1}, ..., \hat{v}_i..., v_{2r+1})$ where $v_{2r+1} = w_{2r+1} + e_{2r+1}$
as in Construction \ref{const:GammaSections} (\ref{const:GammaSections:item4}) with 
$v_1,...,v_{2r},w_{2r+1} \in R^{2r}$.
Since $e_{2r+1}$ is perpendicular to $R^{2r}$, the two sequences $(v_{1}, ..., \hat{v}_i..., v_{2r+1})$ and $(v_{1}, ..., \hat{v}_i..., w_{2r+1})$ have the same (invertible) Gram matrix.
In particular, the latter sequence defines a basis of $R^{2r}$.
It follows that $v_+=(v_{1}, ..., \hat{v}_i..., v_{2r+1}, e_{2r+1})$ is a basis of $R^{2r+1}$.
Write $V_q(d_iA)=(u_{1}, .., u_{2r})$.
By Construction  \ref{const:GammaSections}, this is a basis of $R^{2r}$.
It follows that $v_+$ and $u_+=(u_{1}, .., u_{2r},e_{2r+1})$ both define a basis of $R^{2r+1}$.
Since $\Gamma(v_+) = (d_iA) \ast 0 = \Gamma(u_+)$, the linear endomorphism of $R^{2r+1}$ that sends $v_+$ to $u_+$ is an isometry. 
By Corollary  \ref{cor:WittThm}, this isometry extends to an isometry of 
$R^{2r+2}$ which we can extend to an isometry $B: R^{2n} \to R^{2n}$ such that $Be_i=e_i$, $i=2r+3,...,2n$.
Note that $B(d_iV_{q+1}(A)) = V_{q}(d_iA)$ and $Be_{2r+1}=e_{2r+1}$.
Now, this $B\in \Sp_{2n}(R)$ and any $C\in Sp_{2n-2r-1}(R) \subset \Sp_{2n}(R)$ have matrix representations
$$B = \left(\begin{matrix} M & 0 & u & 0 \\ ^t\!v & 1 & b & 0 \\ 0 & 0 & 1 & 0\\ 0 & 0 & 0 & 1_{2n-2r-2}\end{matrix}\right), \hspace{2ex}
C = \left(\begin{matrix} 1_{2r} & 0 & 0 & 0 \\ 0 & 1 & c & ^t\!y \\ 0 & 0 & 1 & 0 \\ 0 & 0 & x & N\end{matrix}\right)\  \in \Sp_{2n}(R)
$$
where
$M\in \Sp_{2r}(R)$, $N\in \Sp_{2n-2r-2}(R)$, $x,y \in R^{2n-2r-2}$, $u,v\in R^{2r}$ and $b,c\in R$.
Any two such matrices commute because
$$\left(\begin{smallmatrix} M & 0 & u & 0 \\ ^t\!v & 1 & b & 0 \\ 0 & 0 & 1 & 0\\ 0 & 0 & 0 & 1\end{smallmatrix}\right)
\left(\begin{smallmatrix} 1 & 0 & 0 & 0 \\ 0 & 1 & c & ^t\!y \\ 0 & 0 & 1 & 0 \\ 0 & 0 & x & N\end{smallmatrix}\right)
=
\left(\begin{smallmatrix} M & 0 & u & 0 \\ ^t\!v & 1 & b+c & ^t\!y \\ 0 & 0 & 1 & 0\\ 0 & 0 & x & N\end{smallmatrix}\right)
=
\left(\begin{smallmatrix} 1 & 0 & 0 & 0 \\ 0 & 1 & c & ^t\!y \\ 0 & 0 & 1 & 0 \\ 0 & 0 & x & N\end{smallmatrix}\right)
\left(\begin{smallmatrix} M & 0 & u & 0 \\ ^t\!v & 1 & b & 0 \\ 0 & 0 & 1 & 0\\ 0 & 0 & 0 & 1\end{smallmatrix}\right)
$$
This finishes the proof.
\end{proof}

\begin{corollary}
\label{cor:dp20}
For $n\geq1 $ and $p\geq 0$, the following differential is trivial:
$$0 = d^1_{p,2}:E^1_{p,2}(R^{2n}) \to E^1_{p,1}(R^{2n}).$$
\end{corollary}

\begin{proof}
By Lemma \ref{lem:commdiag}, this differential is the map
$$ \eps_*\otimes d: H_*(\Sp_{2n-2}(R)) \otimes_{\Z} \Z[\Skew_{2}^+(R)] \longrightarrow H_*(\Sp_{2n-1}(R)) \otimes_{\Z} \Z[\Skew_{1}^+(R)]$$
But $d: \Z[\Skew_{2}^+(R)] \to  \Z[\Skew_{1}^+(R)]$ is the zero map.
\end{proof}

\section{Triviality of $d^r_{p,q}$ for $q$ even}

The goal in this section is to show that the differentials $d_{p,q}^r$ of the spectral sequence (\ref{SpecSeq}) vanish for $r\geq 2$ and $q<2n$ even.
Since $d\circ d=0$ in the complex $C_*(R^{2n})$, the diagram
$$\xymatrix{
 0 \ar[rr] \ar[d] && \Z  [U_{2r+1}(R^{2r})] \ar[d]^{d} \ar[r] & 0\ar[d] \\
 \Z[U_{2r+1}(R^{2n})]  \ar[rr]_d &&  \Z[U_{2r}(R^{2n})] \ar[r]_d  & \Z[U_{2r-1}(R^{2n})]}$$
commutes and defines a map 
of complexes $\ffi: \Z[U_{2r+1}(R^{2r})][-2r] \to C_{\leq 2n+1}(R^{2n})$ of $\Sp_{2n-2r}(R)$-modules 
where $\Sp_{2n-2r}(R)$ acts trivially on the source complex $Z  [U_{2r+1}(R^{2r})][-2r]$ and via its inclusion $\eps:\Sp_{2n-2r}(R) \subset \Sp_{2n}(R)$ on the target complex.
The pair 
$$(\eps,\ffi):(\Sp_{2n-2r}(R),  \Z[U_{2r+1}(R^{2r})][-2r]) \longrightarrow (\Sp_{2n}(R), C_{\leq 2n+1}(R^{2n}))$$
defines a map of associated group homology spectral sequences 
\begin{equation}
\label{eqn:SpSeqMapp}
E_{p,q}^s({2n};r) \longrightarrow E_{p,q}^s({2n})
\end{equation}
resulting from the stupid filtrations of the coefficient complexes; see (\ref{eqn:GpHomSpSeq}).
By definition, we have
$$E_{p,q}^s({2n};r)=
\left\{ 
\renewcommand\arraystretch{1.5}
\begin{array}{ll}
0,&  q\neq 2r\\
\Z[U_{2r+1}(R^{2r})], & q=2r.
\end{array}\right.$$

\begin{proposition}
\label{prop:Esurjective}
For integers $0 \leq r < n$, $s=2$, and all $0 \leq p$, the map (\ref{eqn:SpSeqMapp}) is surjective in bidegree $(p,2r)$:
$$E_{p,2r}^2({2n};r) \twoheadrightarrow E_{p,2r}^2({2n}).$$
\end{proposition}

\begin{proof}
The map $E^1_{p,2r}({2n};r) \to  E^1_{p,2r}({2n})$ is the diagonal arrow in the diagram
$$\xymatrix{
H_p(\Sp_{2n-2r}) \otimes \Z[U_{2r+1}(R^{2r})] \ar[drr] \ar[rr]^{1\otimes d\circ \Gamma} & & H_p(\Sp_{2n-2r}) \otimes  \Z[\Skew_{2r}^+] \ar[d]^{\text{ (\ref{eqn:E1ShapiroIdentn})}} \\
&& H_p(\Sp_{2n},U_{2r}(R^{2n}))
}
$$
sending $ \alpha\otimes v$ to  $(\eps,dv)_*(\alpha)$.
This diagram commutes because 
$d_iv$ and $V_{2r,2n}(\Gamma d_iv)$ are two basis of the subspace $R^{2r} \subset R^{2n}$ that have the same Gram matrix.
In particular, the automorphism $B:R^{2r} \to R^{2r}$ sending $d_iv$ to $V_{2r,2n}(\Gamma d_iv)$ is an isometry of $R^{2r}$ which we extend to $B \in \Sp_{2n}(R)$ by requiring $Be_i = e_i$ for $i=2r+1,...,2n$.
Since $B$ commutes with every element in $\Sp_{2n-2r}(R) \subset \Sp_{2n}(R)$ the diagram commutes.
It follows that under the isomorphism (\ref{eqn:E1ShapiroIdentn}), 
the map $E^1_{p,2r}({2n};r) \to  E^1_{p,2r}({2n})$ is the first map in the complex
$$\xymatrix{
H_p(\Sp_{2n-2r})  \otimes \Z[U_{2r+1}]  \ar[r]^{1\otimes d\circ \Gamma} & H_p(\Sp_{2n-2r}) \otimes  \Z[\Skew_{2r}^+] \ar[d]^{\eps_*\otimes d}\\
&   H_p(\Sp_{2n-2r+1})\otimes Z[\Skew_{2r-1}^+].
}$$
In view of Lemma \ref{lem:commdiag}, the second map in that complex is $d^1_{p,2r}: E^1_{p,2r}({2n}) \to E^1_{p,2r-1}({2n})$.
Since $\eps_*:H_p(\Sp_{2n-2r}) \to H_p(\Sp_{2n-2r+1})$ is (split) injective,
Lemmas \ref{lem:GammaBij} and \ref{lem:SkewExact} imply that this complex is exact.
It follows that $E^1_{p,2r}({2n};r)$ surjects onto the kernel of the right vertical map which surjects onto $E^2_{p,2r}({2n})$.
In particular, $E^2_{p,2r}({2n};r)$ surjects onto $E^2_{p,2r}({2n})$.
\end{proof}

\begin{corollary}
\label{cor:d1SurjIso}
Let $R$ be a local ring with infinite residue field.
Then for $q<2n$ even and $s\geq 2$, the spectral sequence (\ref{SpecSeq})  satisfies $d_{p,q}^s=0$.
\end{corollary}

\begin{proof}
The spectral sequence $E({2n};r)$ has all differentials $d^{s}_{p,q}=0$, by definition.
Thus, if $E_{p,q}^s({2n};r) \to E^s_{p,q}({2n})$ is surjective, then the differential $d^s_{p,q}$ of $E({2n})$ vanishes and the map $E_{p,q}^{s+1}({2n};r) \to E^{s+1}_{p,q+2r}({2n})$ is surjective. 
Therefore, Proposition \ref{prop:Esurjective} implies that all differentials leaving $E^s_{p,q}({2n})$ vanish for $s\geq 2$ and $q<2n$ even.
\end{proof}

\section{A formula for $d^2_{0,2n+1}$}

Our aim is to show that the differential 
$$d^2_{0,2n+1}:E_{0,2n+1}^2(2n)\to E_{1,2n-1}^2(2n) = H_1(\Sp_{2n},\Z[U_{2n-1}])$$
in the spectral sequence (\ref{SpecSeq}) is surjective at least when $2n=4$.
This will be achieved in Proposition \ref{prop:d2Surjects}.
In this section we will give an explicit formula for this differential in Proposition \ref{prop:dAFormula}.
Note that $E_{1,2n-1}^2(2n) \subset E_{1,2n-1}^1(2n)$ since $E_{1,2n}^1(2n)=0$.

Recall from Lemma \ref{lem:SkewExact} the surjection 
$$d: \Z[\Skew_{2n+2}^+] \twoheadrightarrow E_{0,2n+1}^2(2n) = \ker\left( \Z [\Skew_{2n+1}^+] \stackrel{d}{\longrightarrow} \Z [\Skew_{2n}^{+}]\right).$$
For every matrix $A \in \Skew_{2n-1}^+(R)$ we chose  
$$v(A)=(v_1,...,v_{2n-1}) \in U_{2n-1}(R^{2n-1})\subset GL_{2n-1}(R)$$
 with $\Gamma(v(A)) = A$, $\det v(A)=1$ and $v_i \in R^i$ for $i=1,...,2n-1$ where $R^i$ is considered a subspace of $R^{2n-1}$ via $e_j \mapsto e_j$, $j=1,...,i$.
 This is possible, for if we write the section $V_{2n-1,2n-2}(A)$ of Construction \ref{const:GammaSections} as
 $V_{2n-1,2n-2}(A)=v'(A)=(v_1,...,v_{2n-2},v_{2n-1}')$ then $\Gamma(v'(A))=A$ and 
 $v(A)=(v_1,...,v_{2n-2},v_{2n-1})$ has the required properties where $v_{2n-1}=v'_{2n-1}+\det_{R^{2n-2}}^{-1}(v_1,...,v_{2n-2})\cdot e_{2n-1}$.
 
Using the identification $H_1\Sp_1(R) \cong \Sp_1(R) \cong R: \left(\begin{smallmatrix}1& x \\ 0 & 1\end{smallmatrix}\right)\mapsto x$ of Example \ref{ex:H1G},  Shapiro's Lemma yields the isomorphism
\begin{equation}
\label{eqn:H1GUident}
\renewcommand\arraystretch{1.5}
\begin{array}{rcl}
R\ [\Skew_{2n-1}^+(R)] & \stackrel{\cong}{\longrightarrow} & H_1\left(\Sp_{2n}(R),\Z [U_{2n-1}(R^{2n})]\right)\\
a \cdot [A] & \mapsto &\left[
e_{2n-1,2n}(a), v(A) \right]
\end{array}
\end{equation}
where $e_{ij}(a)$ denotes the standard elementary matrix with $1$'s on the diagonal, $a$ at the $(i,j)$-spot and zero elsewhere, $i\neq j$.

For $A\in \Skew^+_{2n+2}(A)$ and subset $I\subset \{1,2,...,2n+2\}$, we denote by $A^{\wedge}_I$ the skew-symmetric matrix obtained from $A$ by removing all rows and columns in $I$.
Now we can state the explicit formula for the differential under consideration.

\begin{proposition}
\label{prop:dAFormula}
Under the isomorphism (\ref{eqn:H1GUident}),  the composition 
{\small
$$\gamma: \Z[\Skew_{2n+2}^+(R)] \stackrel{d}{\twoheadrightarrow} E_{0,2n+1}^2(2n) \stackrel{d^2_{0,2n+1}}{\longrightarrow} E_{1,2n-1}^2(2n) \subset E_{1,2n-1}^1(2n) = R \ [\Skew_{2n-1}^+(R)]$$ 
}
sends a generator $A \in \Skew_{2n+2}^+(R)$ to 
$$\gamma(A) = \sum_{1\leq i<j<k\leq 2n+2}(-1)^{i+j+k} \ \  \frac{\Pf(A)}{\Pf(A^{\wedge}_{ij})\Pf(A^{\wedge}_{ik})\Pf(A^{\wedge}_{jk})} \cdot  \left[A_{\widehat{ijk}}\right].$$
\end{proposition}

\begin{proof}
The spectral sequence $E_{p,q}(2n)$ of (\ref{SpecSeq}) is the spectral sequence 
$$E_{p,q}^2 = H_q(H_p(C,d^h),d^v) \Rightarrow H_{p+q}(\Tot C)$$
 associated with the double complex $C_{p,q}=\Z[G^p\times U_q]$ where $G=\Sp_{2n}(R)$ and $U_q=U_q(R^{2n})$, $q \leq 2n+1$.
Horizontal and vertical differentials $d^h$ and $d^v$ are induced by the differential in the Bar complex $\Z[G^*\times U_q]$ and the complex $\Z[U_*]$, respectively.
For a spectral sequence arising from a double complex as above, the differential $d^2_{p,q}:E^2_{p,q} \to E^2_{p+1,q-2}$ is defined as follows.
An element $[x]\in E^2_{p,q}$ is represented by an element $x\in C_{p,q}$ such that $d^hx=0 \in C_{p-1,q}$ and there is $y\in C_{p+1,q-1}$ such that $d^hy=d^vx\in C_{p,q-1}$.
Then $d^2_{p,q}[x] = [d^vy]\in E^{2}_{p+1,q-2}$.

In our case, let $A\in \Skew_{2n+2}^+(R)$.
Then $d(A)=\sum_{i=1}^{2n+2}(-1)^{i+1}[A^{\wedge}_{i}] \in E^2_{0,2n+1}$ is represented by 
$$\alpha = \sum_{i=1}^{2n+2}(-1)^{i+1}[(u_{i})^{\wedge}_i ] \in \Z[U_{2n+1}(R^{2n})] \in C_{0,2n+1}$$
 where $u_i\in U_{2n+2}(R^{2n})$ is such that $\Gamma(u_i)^{\wedge}_i = A^{\wedge}_{i}$.
To find such $u_i$, use Lemma \ref{lem:GammaBij} to find $(u_i)^{\wedge}_{i}\in U_{2n+1}(R^{2n})$ and then Lemma \ref{lem:goodPosition} to extend it to $u_i \in U_{2n+2}(R^{2n})$.
For $1 \leq j \leq 2n+2$, $j\neq i$, the unimodular sequences $(u_i)^{\wedge}_{ij} \in U_{2n}(R^{2n})\subset GL_{2n}(R)$ are invertible matrices and we set 
$$g_{ji}=(u_j)^{\wedge}_{ij}\circ \left((u_i)^{\wedge}_{ij}\right)^{-1}.$$
Since $\Gamma \left((u_i)^{\wedge}_{ij} \right)= A^{\wedge}_{i,j} = \Gamma (\left(u_j)^{\wedge}_{ij}\right)$, the change of basis matrix $g_{ji}$ is an isometry.
In particular, $g_{ji}\in \Sp_{2n}(R)$, and we set
$$\beta = \sum_{1 \leq i < j \leq {2n+2}} (-1)^{i+j+1}\  [g_{ji},(u_i)^{\wedge}_{ij}] \in \Z[\Sp_{2n}\times U_{2n}] = C_{1,2n}.$$
Using the equality
$$d^h\ [g_{ji}, (u_i)^{\wedge}_{ij}] = (u_i)^{\wedge}_{ij} - g_{ji}\cdot (u_i)^{\wedge}_{ij} = (u_i)^{\wedge}_{ij} - (u_j)^{\wedge}_{ij},$$
we check that $d^h(\beta) = d^v(\alpha)\in C_{0,2n} $:
$$
\renewcommand\arraystretch{3.5}
\begin{array}{rcl}
d^h(\beta) 
& =  & 
\displaystyle\sum_{1 \leq i < j \leq 2n+2} (-1)^{i+j+1}\ (u_i)^{\wedge}_{ij} + \sum_{1 \leq i < j \leq 2n+2} (-1)^{i+j}\ (u_j)^{\wedge}_{ij}\\
& = & 
\displaystyle \sum_{i=1}^{2n+2} \sum_{j=i+1}^{2n+2} (-1)^{i+j+1}\ d_{j-1}\left((u_i)^{\wedge}_i\right)+ \sum_{1 \leq j < i \leq {2n+2}} (-1)^{i+j}\ (u_i)^{\wedge}_{ij}\\
& = & 
 \displaystyle\sum_{i=1}^{2n+2}\left(\sum_{j=i}^{2n+1}(-1)^{i+j}\ d_{j}\left((u_i)^{\wedge}_i\right)+ \sum_{j=1}^{i-1} (-1)^{i+j}\ d_{j}\left((u_i)^{\wedge}_i\right)\right)\\
&=& 
\displaystyle\sum_{j=1}^{2n+1} (-1)^{j+1}\  \sum_{i=1}^{2n+2} (-1)^{i+1}\ d_j \left((u_i)^{\wedge}_i\right)\\
&=&   d^v(\alpha).
\end{array}
$$
Thus, we have 
$$
\renewcommand\arraystretch{3}
\begin{array}{rcl}
\gamma(A)& =&  d^v(\beta) \\\
& =  & 
\displaystyle\sum_{k=1}^{2n}  \ \   \sum_{1 \leq i < j \leq {2n+2}}\ (-1)^{i+j+k}\  [g_{ji},d_k(u_i)^{\wedge}_{ij}] \\
& = & 
\displaystyle\sum_{1 \leq k < i < j \leq {2n+2}}(-1)^{i+j+k}\  [g_{ji},(u_i)_{\widehat{ijk}}] \\
&  & 
- \displaystyle\sum_{1 \leq i < k < j \leq {2n+2}}(-1)^{i+j+k}\  [g_{ji},(u_i)_{\widehat{ijk}}] \\
&  & 
+ \displaystyle\sum_{1 \leq i < j < k \leq {2n+2}}(-1)^{i+j+k}\  [g_{ji},(u_i)_{\widehat{ijk}}] \\
& = & 
\displaystyle\sum_{1 \leq i < j < k \leq {2n+2}}(-1)^{i+j+k}\ 
\left( [g_{kj},(u_j)_{\widehat{ijk}}]  - [g_{ki},(u_i)_{\widehat{ijk}}] + [g_{ji},(u_i)_{\widehat{ijk}}] \right).
\end{array}
$$
Note that $ [g_{kj},(u_j)_{\widehat{ijk}}]  - [g_{ki},(u_i)_{\widehat{ijk}}] + [g_{ji},(u_i)_{\widehat{ijk}}] $ is a cycle for the horizontal differential and thus represents an element of $H_1(\Sp_{2n},U_{2n-1})$.
Using the identity $[gh,u]=[h,u]+[g,hu]$ in $H_1(G,U)$ together with $(u_j)_{\widehat{ijk}} = g_{ji} \cdot (u_i)_{\widehat{ijk}}$, this cycle is
$$
\renewcommand\arraystretch{2}
\begin{array}{rcl}
&&
 [g_{kj},g_{ji} \cdot (u_i)_{\widehat{ijk}}]  - [g_{ki},(u_i)_{\widehat{ijk}}] + [g_{ji},(u_i)_{\widehat{ijk}}] \\
 &=&
 [g_{kj}g_{ji}, (u_i)_{\widehat{ijk}}] - [g_{ji}, (u_i)_{\widehat{ijk}}] - [g_{ki},(u_i)_{\widehat{ijk}}] + [g_{ji},(u_i)_{\widehat{ijk}}]\\
  &=&
 [g_{kj}g_{ji}, (u_i)_{\widehat{ijk}}]  - [g_{ki},(u_i)_{\widehat{ijk}}] \\
  &=&
 [g_{ik}g_{kj}g_{ji}, (u_i)_{\widehat{ijk}}]
 \end{array}
$$
where we also used $[g,u]-[h,u]=[h^{-1}g,u]$ if $gu=hu$, and $g_{ik} = g_{ki}^{-1}$.
Hence, 
$$\gamma(A) = \displaystyle\sum_{1 \leq i < j < k \leq {2n+2}}(-1)^{i+j+k}\ 
 [g_{ik}g_{kj}g_{ji}, (u_i)_{\widehat{ijk}}].
$$
The proposition now follows from Lemma \ref{lem:CycleIsPf} below.
\end{proof}

\begin{lemma}
\label{lem:CycleIsPf}
Let $A\in \Skew^+_{2n+2}(R)$ and $1 \leq i<j<k \leq 2n+2$.
Then the element 
$$[g_{ik}g_{kj}g_{ji}, (u_i)_{\widehat{ijk}}] \in H_1(\Sp_{2n}(R),U_{2n-1}(R^{2n}))$$
is independent of the choice of $u_i, u_j, u_k \in U_{2n+2}(R^{2n})$ as long as
$\Gamma(u_r)^{\wedge}_r = A^{\wedge}_{r}$
and $g_{rs}=(u_r)^{\wedge}_{rs}\circ \left((u_s)^{\wedge}_{rs}\right)^{-1}$ for $r,s \in \{i,j,k\}$.
Under the isomorphism (\ref{eqn:H1GUident}), we have
\begin{equation}
\label{eqn:lem:CycleIsPf}
[g_{ik}g_{kj}g_{ji}, (u_i)_{\widehat{ijk}}] = {\tiny \frac{\Pf(A)}{\Pf(A^{\wedge}_{ij})\Pf(A^{\wedge}_{ik})\Pf(A^{\wedge}_{jk})} }\cdot  \left[A_{\widehat{ijk}}\right].
 \end{equation}
\end{lemma}

\begin{proof}
Let $\tilde{u}_i, \tilde{u}_j, \tilde{u}_k \in U_{2n+2}(R^{2n})$ be another set of unimodular sequences with 
$\Gamma(\tilde{u}_r)^{\wedge}_r = A^{\wedge}_{r}$ and 
set $\tilde{g}_{rs}=(\tilde{u}_r)^{\wedge}_{rs}\circ \left((\tilde{u}_s)^{\wedge}_{rs}\right)^{-1}$ for $r,s \in \{i,j,k\}$.
Since $\Gamma: \Sp_{2n}\backslash U_{2n+1}(R^{2n}) \cong \Skew_{2n+1}^+(R)$ is a bijection, there is $h_r\in \Sp_{2n}(R)$ such that $(\tilde{u}_r)^{\wedge}_r = h_r (u_r)^{\wedge}_r$ for $r\in \{i,j,k\}$.
Then $\tilde{g}_{rs} = h_r\cdot g_{rs}\cdot h_s^{-1}$ and
$$[\tilde{g}_{ik}\tilde{g}_{kj}\tilde{g}_{ji}, (\tilde{u}_i)_{\widehat{ijk}}]   
= [h_i(g_{ik}g_{kj}g_{ji})h_i^{-1}, h_i(u_i)_{\widehat{ijk}}]
=
[g_{ik}g_{kj}g_{ji}, (u_i)_{\widehat{ijk}}]$$
using $[hgh^{-1},hu] = [g,u] \in H_1(G,U)$ for all $h,g\in G$ and $u\in U$ with $gu=u$.
This proves independence of choices.

Next we want to reduce checking equation (\ref{eqn:lem:CycleIsPf}) for all indices $i<j<k$ to checking it just for $(i,j,k)=(2n, 2n+1, 2n+2)$.
To that end, if $k<2n+2$, let $B \in \Skew_{2n+2}^+(R)$ be obtained from $A$ by exchanging $k$-th row and column with the $k+1$-st row and column, that is, let $\sigma$ be the permutation of $\{1,...2n+2\}$ that exchanges $k$ and $k+1$ and fixes everything else, then the entries of $B$ satisfy
$B_{r,s} = A_{\sigma(r),\sigma(s)}$.
Note that $A^{\wedge}_{i,j,k}=B^{\wedge}_{i,j,k+1}$.
Similarly, let $v_i, v_j,v_{k+1}$ be obtained from $u_i, u_j,u_{k}$ by exchanging the $k$-th and $k+1$-st columns.
That is, $v_r=(u_{\sigma(r)})\circ \sigma$ where $\sigma$ also denotes the $2n+2\times 2n+2$ permutation matrix corresponding to the permutation $\sigma$ above, $r \in \{i,j,k+1\}$.
Then $\Gamma(v_r)^{\wedge}_r=B^{\wedge}_{r}$ for $r\in \{i,j,k+1\}$.
The right hand side of (\ref{eqn:lem:CycleIsPf}) doesn't change if we replace $A$ with $B$ and $i,j,k$ with $i,j,k+1$ since $\Pf(A)=-\Pf(B)$, $\Pf(A^{\wedge}_{i,j})=-\Pf(B_{i,j})$, $\Pf(A^{\wedge}_{i,k})=\Pf(B_{i,k+1})$ and $\Pf(A^{\wedge}_{j,k})=\Pf(B_{j,k+1})$.
For the left hand side, write $f_{r,s}=(v_r)^{\wedge}_{r,s}\circ ((v_s)^{\wedge}_{r,s})^{-1}$ where $r,s \in \{i,j,k+1\}$. 
Then $f_{r,s}=g_{\sigma(r),\sigma(s)}$ for $r\neq s \in \{i,j,k+1\}$.
In particular,
$$[f_{i,k+1}f_{k+1,j}f_{ji}, (v_i)_{\widehat{ijk+1}}] =[g_{ik}g_{kj}g_{ji}, (u_i)_{\widehat{ijk}}],$$
and the left hand side has not changed.
Similarly, if $k=2n+2$ and $j<2n+1$ one shows that both sides of (\ref{eqn:lem:CycleIsPf}) remain unchanged if we replace $A$ with the matrix $B$ obtained from $A$ by exchanging the $j$-the and $j+1$st row and column.
Finally, if $(j,k)=(2n+1,2n+2)$ and $i<2n$ we can exchange the $i$-th and $i+1$st row and column in $A$ without changing the sides of the equation (\ref{eqn:lem:CycleIsPf}).

Now we are reduced to checking equation (\ref{eqn:lem:CycleIsPf}) for $(i,j,k)=(2n, 2n+1, 2n+2)$.
So, let $(i,j,k)=(2n, 2n+1, 2n+2)$ and set $I=\{i,j,k\}$.
Recall that $v(A^{\wedge}_{I} )\in GL_{2n-1}(R)$ is an upper triangular matrix of determinant $1$ and thus of the form
$$v(A^{\wedge}_{I}) = \left(\begin{smallmatrix}w & w_{2n-1} \\ 0 & \det^{-1}w\\ 0 & 0\end{smallmatrix}\right) \in M_{2n,2n-1}(R)$$
when the columns are considered as lying in $R^{2n}$.
Note that $w\in GL_{2n-2}(R)$.

For $r\in I$, construct $u_r\in U_{2n+2}(R^{2n})$ as follows. 
Set $(u_r)^{\wedge}_{I}=v(A^{\wedge}_{I})$, extend it first to $(u_r)^{\wedge}_r \in U_{2n+1}(R^{2n})$ with $\Gamma((u_r)^{\wedge}_r)=A^{\wedge}_r$ as in Construction \ref{const:GammaSections} and then extend it to $u_r \in U_{2n+2}(R^{2n})$ using Lemma \ref{lem:goodPosition} and Remark \ref{rmk:goodPisitionk}.
Then $g_{rs}$ fixes $(u_r)^{\wedge}_{I}=v(A^{\wedge}_I)=(u_s)^{\wedge}_I$.
Since the columns of $(u_r)^{\wedge}_{I}=v(A^{\wedge}_I)$ generate $R^{2n-1}$, the map $g_{rs}$ fixes every vector in $R^{2n-1}$.
Therefore, $g_{rs}=e_{2n-1,2n}(c_{rs}) \in \Sp_1(R)\subset \Sp_{2n}(R)$ for unique $c_{rs}\in R$ and all $r,s\in I$.
Note that $c_{rs} = - c_{sr}$ since $g_{rs} =  g_{sr}^{-1}$.
Under the identification (\ref{eqn:H1GUident}) we then have
$$[g_{ik}g_{kj}g_{ji},(u_i)^{\wedge}_{I}] = \left(c_{ik}+c_{kj}+c_{ji}\right)\cdot [A^{\wedge}_I].$$

Let $r,s,t\in I$ be such that $\{r,s,t\} = I$, and let
$u^{r}_s$ denote the $s$-th vector in the unimodular sequence $u^r:=u_r$ and $u^{r}_{q,s}$ its $q$-th row entry.
Then 
$$u^r_s = \left( \begin{smallmatrix} w_s \\ d^r_s \\ \Pf(A^{\wedge}_{rt})\end{smallmatrix}\right)$$
where $d^r_s\in R$, and $w_s\in R^{2n-2}$ is the unique solution to 
$$^{t}w \circ \psi \circ w_s = (A_{q,s})_{1\leq q \leq 2n-2}$$
which expresses part of the equality $\Gamma(u^r)^{\wedge}_r = A^{\wedge}_r$.
The last entry of $u^r_s$ follows from the identity
$$ \Pf(A^{\wedge}_{rt}) = \det (u^r)^{\wedge}_{rt}=  u^r_{2n,s} \cdot \det v(A^{\wedge}_I)=u^r_{2n,s}$$  since the matrix  $(u^r)^{\wedge}_{rt}=(v(A^{\wedge}_I),u^r_s)$ is upper triangular and $\det v(A^{\wedge}_I) =1$.

For $s<t$ we have 
$$(u^r)_{\widehat{2n-1,r}} = \left(\begin{smallmatrix} w & w_s & w_ t \\ 0 & d^r_s & d^r_t \\ 0 & \Pf(A^{\wedge}_{rt}) & \Pf(A^{\wedge}_{rs})\end{smallmatrix}\right),$$
 and therefore,
\begin{equation}
\label{eqn:Pfaffdrs}
\renewcommand\arraystretch{2}
\begin{array}{rcl}
\Pf(A_{\widehat{2n-1,r}}) & = & \det (u^r)_{\widehat{2n-1,r}}  \\
&=& \det w \cdot \left( d^r_s \Pf(A^{\wedge}_{rs}) - d^r_t\Pf(A^{\wedge}_{rt})\right)\\
&=& \Pf(A_{\widehat{2n-1,I}}) \cdot \left( d^r_s \Pf(A^{\wedge}_{rs}) - d^r_t\Pf(A^{\wedge}_{rt})\right),\hspace{4ex} s<t.
\end{array}
\end{equation}

For all $r,s,t$ with $\{r,s,t\}=I$, the equation $u^r_t = g_{rs} (u^s_t)$ yields
$$d^r_t = d^s_t + c_{rs} \Pf(A^{\wedge}_{rs}).$$
Therefore, 
$$\renewcommand\arraystretch{3}
\begin{array}{rcl}
&& c_{ik}+c_{kj}+c_{ji} \\
&=& \frac{d^i_j-d^k_j}{\Pf(A^{\wedge}_{ik})} + \frac{d^k_i-d^j_i}{\Pf(A^{\wedge}_{jk})}  + \frac{d^j_k-d^i_k}{\Pf(A^{\wedge}_{ij})} \\
&=& 
 \frac{d^i_j\Pf(A^{\wedge}_{ij}) -d^i_k\Pf(A^{\wedge}_{ik})}{\Pf(A^{\wedge}_{ij})\Pf(A^{\wedge}_{ik})} 
 +
  \frac{d^k_i\Pf(A^{\wedge}_{ik}) -d^k_j\Pf(A^{\wedge}_{jk})}{\Pf(A^{\wedge}_{ik})\Pf(A^{\wedge}_{jk})} 
+
 \frac{d^j_k\Pf(A^{\wedge}_{jk}) -d^j_i\Pf(A^{\wedge}_{ij})}{\Pf(A^{\wedge}_{jk})\Pf(A^{\wedge}_{ij})} \\
 &\stackrel{\text{(\ref{eqn:Pfaffdrs})}}{=}&
  \frac{\Pf(A_{\widehat{2n-1,i}})}{\Pf(A^{\wedge}_{ij})\Pf(A^{\wedge}_{ik})\Pf(A_{\widehat{2n-1,I}})} 
  +
    \frac{\Pf(A_{\widehat{2n-1,k}})}{\Pf(A^{\wedge}_{ik})\Pf(A^{\wedge}_{jk})\Pf(A_{\widehat{2n-1,I}})} 
    -
      \frac{\Pf(A_{\widehat{2n-1,j}})}{\Pf(A^{\wedge}_{ij})\Pf(A^{\wedge}_{jk})\Pf(A_{\widehat{2n-1,I}})} \\
      &=&
   \frac{\Pf(A_{\widehat{2n-1,i}})\Pf(A^{\wedge}_{j,k})- \Pf(A_{\widehat{2n-1,j}})\Pf(A^{\wedge}_{i,k}) + \Pf(A_{\widehat{2n-1,k}})\Pf(A^{\wedge}_{i,j})}{\Pf(A^{\wedge}_{ij})\Pf(A^{\wedge}_{ik})\Pf(A^{\wedge}_{j,k})\Pf(A_{\widehat{2n-1,I}})} \\
   &=&
   \frac{\Pf(A)\Pf(A_{\widehat{2n-1,I}})}{\Pf(A^{\wedge}_{ij})\Pf(A^{\wedge}_{ik})\Pf(A^{\wedge}_{j,k})\Pf(A_{\widehat{2n-1,I}})}
\end{array}
$$
where the last equation follows from \cite[Theorem 1]{DressWenzel} (in the notation of that paper, choose $I_1 = \{1,2,...,2n-2,2n,2n+1,2n+2\}$ and $I_2=\{1,2,...,2n-1\}$).
\end{proof}

\section{Surjectivity of $d^2_{0,5}$}

The goal of this section is to show in Proposition \ref{prop:d2Surjects} below that the map
$\gamma:\Z[\Skew_6^+(R)] \to R\ [\Skew_3^+(R)]$ of Proposition \ref{prop:dAFormula} is surjective. 
Recall that, unless stated otherwise, $(R,m)$ is a local ring with infinite residue field $R/m$.

\begin{notation}
A skew-symmetric matrix $A=(a_{ij})$ with entries in $R$ will be specified by giving its upper triangular part, the lower triangular part being determined by the requirement $a_{ij}=-a_{ji}$.
For instance, 
$$\left(\begin{smallmatrix}0 & a & b\\ & 0 & c \\ && 0 \end{smallmatrix}\right) = \left(\begin{smallmatrix}0 & a & b\\ -a & 0 & c \\ -b & -c & 0 \end{smallmatrix}\right).$$
For $a,b,c \in R^*$, we set
$$\left[ \begin{smallmatrix}a & b \\ &  c \end{smallmatrix}\right] = \left( \begin{smallmatrix}0 & a & b\\ & 0 & c \\ && 0 \end{smallmatrix}\right)  \hspace{3ex}\text{and}\hspace{3ex} [a]= \left[ \begin{smallmatrix}a & a \\ &  a \end{smallmatrix}\right]$$
These are elements of $\Skew_3^+(R)$.
For units $a,b,c \in R^*$ with $a^{-1}-b^{-1}+c^{-1}\in R^*$ and $x\in R$
we write 
$$x\ \left\{\begin{smallmatrix} a&b\\   &c\end{smallmatrix}\right\}  =\frac{x}{(a^{-1}-b^{-1}+c^{-1})^2}  \ \left[\begin{smallmatrix} a^{-1}&b^{-1}\\   &c^{-1}\end{smallmatrix}\right]$$
and for $a\in R^*$, $x\in R$ we set
$$x\ \{a\} =x\  \left\{\begin{smallmatrix} a&a\\   &a\end{smallmatrix}\right\} = a^2x\ [a^{-1}]$$
considered as elements of $R\ [\Skew_3^+(R)]$.
Note that 
$$x\ \left\{\begin{smallmatrix} a&a\\   &c\end{smallmatrix}\right\}  =xc^2  \ \left[\begin{smallmatrix} a^{-1}&a^{-1}\\   &c^{-1}\end{smallmatrix}\right]$$
for all $a,c\in R^*$.
Finally, for $\xi,\zeta\in R\ [\Skew_3^+(R)]$, we write 
$$\xi \equiv \zeta$$
 to mean $\xi = \zeta$ in $\coker(\gamma)$ where $\gamma:\Z[\Skew_6^+(R)] \to R\ [\Skew_3^+(R)]$ is the map in Proposition \ref{prop:dAFormula} for $n=2$.
 Our goal is to show $\xi \equiv 0$ for all $\xi\in R\ [\Skew_3^+(R)]$.
\end{notation}

\begin{lemma}
\label{lem:dAexample1aRewrite}
For all $a,b,c,d,e\in R^*$ we have in $R\ [\Skew_3^+(R)]$ the relation
\begin{equation}
\label{eqn:dAexample1aRewrite}
\renewcommand\arraystretch{3}
\begin{array}{rcl}
0 & \equiv &
  \frac{a^2}{b^2} \frac{(d^2-c^2)}{e} \left\{\begin{smallmatrix} a&a\\   &b\end{smallmatrix}\right\}
- \frac{a^2d^2}{c^2e} \left\{\begin{smallmatrix} a&a\\   &c\end{smallmatrix}\right\}
-  \frac{a^2}{c}\left\{\begin{smallmatrix} a&a\\   &e\end{smallmatrix}\right\}
+  \frac{b^2d^2}{c^2e}\left\{\begin{smallmatrix} b&b\\   &c\end{smallmatrix}\right\}  \\
&&
+\  \frac{b^2}{c}  \left\{\begin{smallmatrix} b&b\\   &e\end{smallmatrix}\right\} 
+   b\left\{\begin{smallmatrix} d&d\\   &e\end{smallmatrix}\right\}
- b\left\{\begin{smallmatrix} c&c\\   &e\end{smallmatrix}\right\}.
\end{array}
\end{equation}
\end{lemma}

\begin{proof}
The $6\times 6$ skew-symmetric matrix
$$A = \left( \begin{array}{ccc|ccc}
0 & a&  a& a& a& a \\
   & 0 & b  & b  & b  & b\\
    &  & 0 & c & c& c\\
    \hline
     & &    & 0 & d & d\\
     & &    &    & 0 & e\\
     &&&&& 0
     \end{array}\right)$$
has Pfaffian $\Pf(A)=ace$ and is in $\Skew_6^+(R)$ if and only if all its entries $a,...,e$ are units.
We compute
\begin{equation}
\label{eqn:dAexample1a}
\renewcommand\arraystretch{2}
\begin{array}{rcl}
\gamma(A) & = &\left( \frac{e}{a^2d^2} - \frac{e}{a^2c^2}\right)\  \left[\begin{smallmatrix} a&a\\   &b\end{smallmatrix}\right] \\
&&- \frac{e}{a^2d^2}\ \left[\begin{smallmatrix} a&a\\   &c\end{smallmatrix}\right] 
-\frac{c}{a^2e^2}  \left[\begin{smallmatrix} a&a\\   &e\end{smallmatrix}\right] 
+\frac{e}{b^2d^2}\  \left[\begin{smallmatrix} b&b\\   &c\end{smallmatrix}\right] \\
&&
+ \frac{c}{b^2e^2}\  \left[\begin{smallmatrix} b&b\\   &e\end{smallmatrix}\right]
+  \frac{1}{be^2}\left[\begin{smallmatrix} d&d\\   &e\end{smallmatrix}\right] 
-   \frac{1}{be^2} \left[\begin{smallmatrix} c&c\\   &e\end{smallmatrix}\right].
\end{array}
\end{equation}
See Appendix \ref{App:lem:dAexample1aRewrite} for more details. The result follows by replacing $a,...,e$ with their inverses $a^{-1},...,e^{-1}$.
\end{proof}

\begin{lemma}
\label{lem:dAexample1b}
For all $a,c,d,e\in R^*$ we have in $R\ [\Skew_3^+(R)]$ the relation
 \begin{equation}
\label{eqn:dAexample1b2}
\renewcommand\arraystretch{2}
\begin{array}{l}
 \frac{d^2-c^2}{e}\  \{a\}  \equiv
a\  \left(\left\{\begin{smallmatrix} c&c\\   &e\end{smallmatrix}\right\} - \left\{\begin{smallmatrix} d&d\\   &e\end{smallmatrix}\right\}\right).
 \end{array}
 \end{equation}
\end{lemma}

\begin{proof}
Put $a=b$ in Lemma \ref{lem:dAexample1aRewrite}.
\end{proof}

\begin{lemma}
\label{lem:LinearAsLongAsUnits}
The function  $R \times R^* \to \coker(\gamma): (x,a) \mapsto x\ \{a\}$
is linear in both variables $x$ and $a$ (as long as the second variable only involves units), that is,
$$(x+y)\ \{a\} \equiv x\ \{a\} + y\ \{a\} \hspace{2ex}\text{and}\hspace{2ex}x\ \{a+b\} \equiv x\ \{a\} + x\ \{b\} $$
for all $a,b,a+b\in R^*$, $x,y\in R$.
\end{lemma}

\begin{proof}
The first relation is actually an equality in $R\ [\Skew_3^+(R)]$:
$$(x+y)\ \{a\} = (x+y) a^2[a^{-1}] = xa^2\ [a^{-1}] + ya^2\ [a^{-1}] = x\ \{a\} + y\ \{a\}.$$
For the second relation, for fixed $c,d,e \in R^*$, the right hand side of (\ref{eqn:dAexample1b2}) is linear in $a$, hence the left hand side is, as long as it is defined.
If $x\in R$ is a unit, choose $c,d\in R^*$ such that $d^2-c^2\in R^*$.
This is possible since for given $c\in R^*$, $d$ only needs to avoid finitely many elements of $R/m$.
Setting $e=(d^2-c^2)x^{-1}$ shows that the function $R^* \to \coker(\gamma):a \mapsto x\{a\}$ is linear in $a$.
It follows that for any two units $x,y\in R$, the expression 
$(x+y)\ \{a\} = x\ \{a\} + y\ \{a\}$ is linear in $a$.
Since every element of $R$ is a sum of two units, we are done.
 \end{proof}

We will extend the map in Lemma \ref{lem:LinearAsLongAsUnits} to a $\Z$-bilinear map defined on all of $R\times R$ using the following.

\begin{lemma}
\label{lem:ExtendLinearly}
Let $(R,m)$ be a local ring with infinite residue field, and let $A$ be an abelian group.
Then any function $f:R^* \to A$ satisfying $f(a)+f(b)=f(a+b)$ for all $a,b,a+b\in R^*$ extends to a unique $\Z$-linear function $R \to A$.
\end{lemma}

\begin{proof}
If $x\in m$ and $a,b\in R^*$ then
$f(x+a)-f(a) = f(x+b)-f(b)$.
For we can choose $c \in R^*$ such that $a+c, b+c, a+b+c \in R^*$ as $c$ only needs to avoid a finite number of elements of $R/m$.
Then $f(x+a) + f(b) + f(c)  = f(x+a+c) + f(b) = f(x+a+b+c) = f(x+b) + f(a+c) = f(x+b) + f(a)+ f(c)$.
Thus, we can set $f(x)=f(a+x)-f(a)$ for $a \in R^*$, and this expression is independent of $a\in R^*$.
This defines a function $f: R \to A$ which we need to check is $\Z$-linear.
So, let $x, y\in R$.
If $x,y\in m$ then choose $a,b\in R^*$ such that $a+b\in R^*$.
Then $f(x+y)= f(x+y+a+b)-f(a+b) = f(x+a) + f(y+b) - f(a)-f(b) = f(x)+f(y)$.
If $x\in m$ and $y\in R^*$ then $f(x+y) = f(x)+f(y)$, by definition of $f(x)$.
The rest is clear.
\end{proof}

\begin{definition}
\label{dfn:BilinMap}
We define the map $R \times R \to \coker(\gamma): (x,y) \mapsto \langle x, y\rangle$  by
$$\langle x, y\rangle = \left\{
\renewcommand\arraystretch{1.5}
\begin{array}{cl}x\{y\} & y \in R^*,\\ x\ \{y + a\} - x\ \{a\} & y\in m, a\in R^*.\end{array}\right.$$
By Lemmas \ref{lem:LinearAsLongAsUnits} and \ref{lem:ExtendLinearly}, this is well-defined (that is, independent of the choice of $a\in R^*$), and the map is $\Z$-linear in both variables $x$ and $y\in R$. 
\end{definition}

\begin{lemma}
\label{lem:IndepExp}
For $c,e\in R^*$, the expression 
$$
a^2e^{-1}\ \left\{\begin{smallmatrix}a & a \\ & c\end{smallmatrix}\right\}\ + \
a^2c^{-1}\ \left\{\begin{smallmatrix}a & a \\ & e\end{smallmatrix}\right\}
$$
in $\coker(\gamma)$ is independent of $a\in R^*$.
\end{lemma}

\begin{proof}
The claim follows by setting $c=d$ in Lemma \ref{lem:dAexample1aRewrite}
\end{proof}

\begin{lemma}
\label{lem:aacTof}
For all $c,d, f\in R^*$ we have 
$$
f\,  \left\{\begin{smallmatrix} c&c\\   &d\end{smallmatrix}\right\} \ \equiv\
d\, \{f\}    + f\, \{d\}  -  c^2d^{-1}\{f\} .$$
\end{lemma}

\begin{proof}
This is Lemma \ref{lem:dAexample1b} with $e=d$ replacing $a$ with $f$.
\end{proof}

\begin{lemma}
\label{lem:UnitAntiSym}
For all $a,c,e\in R^*$ we have in $\coker(\gamma)$
$$ 0\ \equiv \  \left\langle a^2c^{-1}- c,\ a^2{e^{-1}} - e\right\rangle
 +  \left\langle a^2{e^{-1}} - e,\, a^2c^{-1} - c\right\rangle.$$
\end{lemma}

\begin{proof}
We rewrite the expression in Lemma \ref{lem:IndepExp} using Lemma \ref{lem:aacTof}:
\begin{equation}
\label{eqn:dAexample1bbCont}
\renewcommand\arraystretch{2}
\begin{array}{rl}
&
 \frac{a^2}{e}\, \left\{\begin{smallmatrix} a&a\\   &c\end{smallmatrix}\right\} + \frac{a^2}{c}\, \left\{\begin{smallmatrix} a&a\\   &e\end{smallmatrix}\right\}  \\
\equiv &\frac{a^2}{e}\,  \{c\} + c\, \{\frac{a^2}{e}\} - \frac{a^2}{c} \, \{\frac{a^2}{e}\}+ 
\frac{a^2}{c}\, \{e\} + e\, \{\frac{a^2}{c}\} - \frac{a^2}{e}\, \{\frac{a^2}{c}\}\\
=&\frac{a^2}{e}\, \{c\} +\frac{a^2}{c}\, \{e\} + \left(e-\frac{a^2}{e}\right)\, \{\frac{a^2}{c}\} +  \left(c-\frac{a^2}{c}\right)\, \{\frac{a^2}{e}\}.
\end{array}
\end{equation}
By Lemma \ref{lem:IndepExp}, these expressions are independent of $a$.
Setting $a=e$, the last expression equals
$$e\, \{c\} + c\, \{e\} - \frac{e^2}{c}\, \{e\} + 
\frac{e^2}{c}\, \{e\} + e\, \{\frac{e^2}{c}\} - e\, \{\frac{e^2}{c}\} =  e\, \{c\} + c\, \{e\}.$$
Taking the difference with (\ref{eqn:dAexample1bbCont}) gives
\begin{equation}
\label{eqn:dAexample1bbCont2}
\renewcommand\arraystretch{2}
\begin{array}{rcl}
 0 & \equiv & - \left(e-\frac{a^2}{e}\right)\, \{c\} -\left(c-\frac{a^2}{c}\right)\,  \{e\}
 +\left(e-\frac{a^2}{e}\right)\,  \{\frac{a^2}{c}\} + \left(c-\frac{a^2}{c}\right)\,  \{\frac{a^2}{e}\}\\
 &=& \left(e-\frac{a^2}{e}\right) \, \left( \{\frac{a^2}{c}\} - \{c\}\right)
+  \left(c-\frac{a^2}{c}\right) \, \left( \{\frac{a^2}{e}\} - \{e\}\right).
\end{array}
\end{equation}
Rewriting the last expression in terms of $\langle\phantom{a},\phantom{b}\rangle$ using Definition \ref{dfn:BilinMap} and bilinearity yields the lemma.
\end{proof}

\begin{lemma}
\label{lem:InversesGenerate}
Let $R$ be a local ring with infinite residue field.
Then the set of elements of the form $\frac{1}{u_1}+\frac{1}{u_2}+\frac{1}{u_3}$
generates $R$ as an abelian group where $u_i\in R^*$ such that $u_1+u_2+u_3=0$.
\end{lemma}

\begin{proof}
Let $A \subset R$ be the abelian subgroup generated by elements of the form $\frac{1}{u_1}+\frac{1}{u_2}+\frac{1}{u_3}$ as in the lemma.
We need to show $A=R$.
We will show that there is a triple $(u_1,u_2,u_3)$ such that 
$u_1+u_2+u_3=0$ and $w=\frac{1}{u_1}+\frac{1}{u_2}+\frac{1}{u_3} \in R^*$ is a unit.
Then $tw \in A$ for all $t\in R^*$, in particular, $R^*\subset A$ which implies $R \subset A$.
We need to find $u_1,u_2\in R^*$ such that $u_1+u_2, \frac{1}{u_1}+\frac{1}{u_2}-\frac{1}{u_1+u_2}\in R^*$ that is, such that $u_1,u_2,u_1+u_2, u_1^2+u_1u_2+u_2^2 \in R^*$.
This only needs to be checked for $F=R/m$.
For any given $0\neq u_1 \in F$ there are only finitely many $u_2$ which do not satisfy this requirement. Since $F$ is infinite, there is $u_2 \in F$ such that 
$u_2,u_1+u_2, u_1^2+u_1u_2+u_2^2 \neq 0$.
\end{proof}

\begin{lemma}
\label{lem:a2overcminusc}
Let $R$ be a local ring with infinite residue field.
Then for  $a\in R^*$ 
the following map of abelian groups (defined on generators) is surjective
$$\Z[R^*] \to R: c \mapsto \frac{a^2}{c}-c.$$
\end{lemma}

\begin{proof}
The image contains the set of elements $a^2(\frac{1}{u_1}+\frac{1}{u_2}+\frac{1}{u_3})$ 
where $u_i\in R^*$ and $u_1+u_2+u_3=0$.
By lemma \ref{lem:InversesGenerate}, this set generates $R$ as an abelian group.
\end{proof}

\begin{lemma}
\label{lem:SymbolSkewSymmetric}
Let $R$ be a local ring with infinite residue field.
Then the  $\Z$-bilinear map $R\times R \to \coker(\gamma): (x,y) \mapsto \langle x, y \rangle$ of Definition \ref{dfn:BilinMap} is anti-symmetric, that is, for all $x,y\in R$ we have
$$\langle x,y\rangle  + \langle y, x \rangle \equiv 0.$$
In particular, for all $a,b\in R^*$ we have 
$$a\, \{b\} + b\, \{a\}\equiv 0.$$
\end{lemma}

\begin{proof}
This follows from Lemmas \ref{lem:a2overcminusc} and \ref{lem:UnitAntiSym}, and the bilinearity of $\langle \phantom{x},\phantom{y}\rangle$.
\end{proof}

\begin{lemma}
\label{lem:aacTof2}
For all $c,d,f\in R^*$ we have
$$
f\,  \left\{\begin{smallmatrix} c&c\\   &d\end{smallmatrix}\right\} \ \equiv\
 -  c^2d^{-1}\{f\}  \equiv f\, \{c^2d^{-1}\}.$$
\end{lemma}

\begin{proof}
Combine Lemmas \ref{lem:aacTof} and \ref{lem:SymbolSkewSymmetric}.
\end{proof}

\begin{lemma}
\label{lem:AlmostThere}
For all $a,b,c\in R^*$ we have
\begin{equation}
\begin{array}{rcl}
0 
&\equiv & a^2\left(\frac{c^2}{b^3} + \frac{1}{c}\right)\, \{\frac{a^2}{b}\}
- b^2\left(\frac{c^2}{b^3}+\frac{1}{c}\right) \{b\}.
\end{array}
\end{equation}
\end{lemma}

\begin{proof}
We simplify expression (\ref{eqn:dAexample1aRewrite}) using Lemma \ref{lem:aacTof2}, and we obtain

\begin{equation}
\label{eqn:dAexample1aRewrite2}
\renewcommand\arraystretch{2}
\begin{array}{rcl}
0 & \equiv &  \frac{a^2}{b^2} \frac{(d^2-c^2)}{e} \, \{\frac{a^2}{b}\}
-  \frac{a^2d^2}{c^2e}\, \{\frac{a^2}{c}\}
-  \frac{a^2}{c}\, \{\frac{a^2}{e}\}
+ \frac{b^2d^2}{c^2e}\, \{\frac{b^2}{c}\} \\
&&
+  \frac{b^2}{c}\, \{\frac{b^2}{e}\}
+ b\, \{\frac{d^2}{e}\}  
- b\, \{\frac{c^2}{e}\} .
\end{array}
\end{equation}
Using anti-symmetry of the bilinear form $\langle \phantom{x},\phantom{y}\rangle$ (Lemma \ref{lem:SymbolSkewSymmetric}) on the 3rd, 5th, 6th and 7th summand yields
\begin{equation}
\label{eqn:dAexample1aRewrite3}
\renewcommand\arraystretch{2}
\begin{array}{rcl}
0 & \equiv & \frac{a^2}{b^2} \frac{(d^2-c^2)}{e} \, \{\frac{a^2}{b}\} 
- \frac{a^2d^2}{c^2e}\, \{\frac{a^2}{c}\} 
+  \frac{a^2}{e}\, \{\frac{a^2}{c}\}
+ \frac{b^2d^2}{c^2e}\, \{\frac{b^2}{c}\} \\
&&
-  \frac{b^2}{e}\, \{\frac{b^2}{c}\}
- \frac{d^2}{e}\,  \{b\} 
+ \frac{c^2}{e}\, \{b\} \\
&=&
\frac{d^2-c^2}{e}\, \left( \frac{a^2}{b^2} \, \{\frac{a^2}{b}\}
-  \frac{a^2}{c^2}\, \{\frac{a^2}{c}\}
+ \frac{b^2}{c^2}\, \{\frac{b^2}{c}\} 
- \{b\}  \right) .
\end{array}
\end{equation}
For any $c\in R^*$ we can find $d\in R^*$ such that $d^2-c^2\in R^*$.
Since $e\in R^*$ is an arbitrary unit we conclude that 
\begin{equation}
\label{eqn:dAexample1aRewrite4}
\renewcommand\arraystretch{2}
\begin{array}{rcl}
0 & \equiv & x\, \left( \frac{a^2}{b^2} \, \{\frac{a^2}{b}\}
-  \frac{a^2}{c^2}\, \{\frac{a^2}{c}\}
+ \frac{b^2}{c^2}\, \{\frac{b^2}{c}\} 
- \{b\}  \right).
\end{array}
\end{equation}
for all $a,b,c\in R^*$ and $x\in R$.
In particular for $x=c^2/b$ we obtain
\begin{equation}
\label{eqn:dAexample1aRewrite5}
\renewcommand\arraystretch{2}
\begin{array}{rcl}
0 & \equiv & 
\frac{a^2c^2}{b^3} \, \{\frac{a^2}{b}\} 
- \frac{a^2}{b}\,  \{\frac{a^2}{c}\}
+ b\, \{\frac{b^2}{c}\} 
-  \frac{c^2}{b} \, \{b\}
\end{array}
\end{equation}
Using anti-symmetry on the 2nd and 3rd term (Lemma \ref{lem:SymbolSkewSymmetric}) we obtain
\begin{equation}
\label{eqn:dAexample1aRewrite6}
\renewcommand\arraystretch{2}
\begin{array}{rcl}
0 & \equiv & 
 \frac{a^2c^2}{b^3} \, \{\frac{a^2}{b}\}
+ \frac{a^2}{c}\,  \{\frac{a^2}{b}\}
-  \frac{b^2}{c}\, \{b\}
-  \frac{c^2}{b}\, \{b\} \\
&=& a^2\left(\frac{c^2}{b^3} + \frac{1}{c}\right)\, \{\frac{a^2}{b}\}
-  b^2\left(\frac{c^2}{b^3}+\frac{1}{c}\right) \, \{b\}.
\end{array}
\end{equation}
\end{proof}

\begin{lemma}
\label{lem:RgenerationX}
Let $R$ be a local ring with infinite residue field and $b\in R^*$ a unit.
Then the following map of abelian groups (defined on generators) is surjective
$$\Z[R^*] \twoheadrightarrow R: c \mapsto \frac{c^2}{b^3}+\frac{1}{c}.$$
\end{lemma}

\begin{proof}
Chose units $u_1,u_2,u_3$ such that all non-empty partial sums are units as well as
$$ w=\frac{1}{u_1+u_2+u_3} - \frac{1}{u_1+u_2} - \frac{1}{u_1+u_3} - \frac{1}{u_2+u_3} + \frac{1}{u_1} + \frac{1}{u_2}+ \frac{1}{u_3} \in R^*.$$ 
This is possible since we only need to find such $u_1,u_2,u_3 \in R/m$ and for any $u_1,u_2 \in R/m$ such that $u_1,u_2,u_1+ u_2\neq 0 \in R/m$ there are only finitely many $u_3 \in R/m$ such that $u_1+u_3$ or $u_2+u_3$, or $u_1+u_2+u_3$ or $w$ is zero in $R/m$ (for $w$ to be zero in $R/m$, $u_3$ has to be a solution of a non-zero polynomial in $R/m$).

Now, let $u_1,u_2,u_3 \in R$ as above. 
The element
$$s=(u_1+u_2+u_3)^2 - (u_1+u_2)^2 - (u_1+u_3)^2 - (u_2+u_3)^2 + u_1^2 + u_2^2+ u_3^2$$
of $R$ is zero.
Hence, $t^2sb^{-3}+t^{-1}w = t^{-1}w$ is in the image of the map for any $t\in R^*$.
Hence $R^*$ is in the image and so is $R$.
\end{proof}

\begin{lemma}
\label{lem:TrivSquareAction}
For all $a\in R^*$ and $ x,y \in R$.
$$\langle a^2 x, y\rangle \equiv \langle x, a^{-2}y\rangle.$$
\end{lemma}

\begin{proof}
Using Lemma \ref{lem:RgenerationX}, Lemma \ref{lem:AlmostThere} implies
\begin{equation}
\label{eqn:dAexample1aRewrite7}
\renewcommand\arraystretch{2}
\begin{array}{rcl}
 a^2x \, \{\frac{a^2}{b}\}
& \equiv &b^2x\,  \{b\} 
\end{array}
\end{equation}
for all $a,b\in R^*$ and $x\in R$ in view of Lemma \ref{lem:RgenerationX}.
In particular, the left hand side is independent of $a$ (as the right hand side is) and thus equals its value for $a=1$. 
In other words (replacing $b$ with $b^{-1}$), we have
\begin{equation}
\label{eqn:dAexample1aRewrite8}
\renewcommand\arraystretch{2}
\begin{array}{rcl}
a^2x\, \{a^2b\} 
& \equiv & x\, \{b\}.
\end{array}
\end{equation}
Replacing $b$ with $a^{-2}b$ yields
$a^2 x\, \{b\} 
 \equiv x\,  \{a^{-2}b\}
$
for all $a,b\in R^*$ and $x \in R$.
That is, the Lemma holds for $y=b$ a unit. 
Since $\langle\phantom{x},\phantom{y}\rangle$ is linear in the second variable, we are done. 
\end{proof}

\begin{lemma}
\label{lem:x[a]iszero}
For all $x\in R$ and $a\in R^*$, we have $x\, [a]\equiv 0$.
\end{lemma}

\begin{proof}
As in the proof of Lemma \ref{lem:RgenerationX}, choose units $u_1,u_2,u_3$ such that all non-empty partial sums are units as well as
$$ w=\frac{1}{(u_1+u_2+u_3)^2} - \frac{1}{(u_1+u_2)^2} - \frac{1}{(u_1+u_3)^2} - \frac{1}{(u_2+u_3)^2} + \frac{1}{u_1^2} + \frac{1}{u_2^2}+ \frac{1}{u_3^2} \in R^*.$$
The element
$$s=(u_1+u_2+u_3)^2 - (u_1+u_2)^2 - (u_1+u_3)^2 - (u_2+u_3)^2 + u_1^2 + u_2^2+ u_3^2 $$
of $R$ is zero.
Therefore, 
$$0 \equiv \langle sx,y \rangle \equiv \langle x, wy\rangle$$
by bilinearity of $\langle \phantom{x},\phantom{y}\rangle$ and Lemma \ref{lem:TrivSquareAction}.
Since $y\in R$ is arbitrary and $w\in R^*$, we have $\langle x, y \rangle=0$ for all $x,y\in R$ which translates into the statement of the Lemma.
\end{proof}

\begin{lemma}
\label{lem:x[aab]iszero}
For all $x\in R$ and $a,b\in R^*$, we have $x\, \left[\begin{smallmatrix}a &a \\ &b\end{smallmatrix}\right]\equiv 0$.
\end{lemma}

\begin{proof}
This is Lemma \ref{lem:x[a]iszero} together with Lemma \ref{lem:aacTof2}.
\end{proof}

\begin{lemma}
\label{lem:x[abc]iszeroConditionally}
For all $x\in R$ and $a,b,c\in R^*$ such that $a-b+c \in R^*$, we have $x\, \left[\begin{smallmatrix}a &b \\ &c\end{smallmatrix}\right]\equiv 0$.
\end{lemma}

\begin{proof}
The skew symmetric matrix
$$A = \left( \begin{array}{ccc|ccc}
0 & d&  d& d& d& d \\
   & 0 & d  & d  & d  & d\\
    &  & 0 & d & d& d\\
    \hline
     & &    & 0 & a & b\\
     & &    &    & 0 & c\\
     &&&&& 0
     \end{array}\right)$$
has Pfaffian $\Pf(A)=d^2(a-b+c)$ and is in $\Skew_6^+(R)$ if and only if $a,b,c,d,(a-b+c) \in R^*$.
Modulo terms that are $\equiv 0$ by Lemmas \ref{lem:x[aab]iszero} and \ref{lem:x[a]iszero}, we have
$$\begin{array}{ccc}\gamma(A) & \equiv & \frac{1}{(a-b+c)^2d}\, \left[\begin{smallmatrix} a&b\\   &c\end{smallmatrix}\right];\end{array}
$$
see Appendix \ref{App:lem:x[abc]iszeroConditionally}. 
Since $d$ is an arbitrary unit in $R$, we are done.
\end{proof}

\begin{lemma}
\label{lem:x[abc]iszero}
For all $x\in R$ and $a,b,c\in R^*$, we have $x\, \left[\begin{smallmatrix}a &b \\ &c\end{smallmatrix}\right]\equiv 0$.
\end{lemma}

\begin{proof}
The $6\times 6$ skew symmetric matrix
$$A = \left( \begin{array}{ccc|ccc}
0 & a&  b& d& d& d \\
   & 0 & b  & e  & e  & e\\
    &  & 0 & f & f& f\\
    \hline
     & &    & 0 & a & b\\
     & &    &    & 0 & b\\
     &&&&& 0
     \end{array}\right)$$
has Pfaffian $\Pf(A)=(af-be+bd)a$ and is in $\Skew_6^+(R)$ if and only if $\Pf(A)$, $a,b,d,e,f$ are units.
Given $a,b,d,f\in R^*$ arbitrary, choose $e$ such that $a-d+e, b-e+f, af-be+bd \in R^*$.
This is possible since $R/m$ is infinite.
Modulo terms that are $\equiv 0$ by Lemmas \ref{lem:x[aab]iszero}, \ref{lem:x[a]iszero} and \ref{lem:x[abc]iszeroConditionally}, we have
$$
\begin{array}{rcccl}
0 &\equiv & \gamma(A) &\equiv&
 -\frac{a}{b^4}\, \left[\begin{smallmatrix} b&d\\   &f\end{smallmatrix}\right];\end{array}
$$
see Appendix \ref{App:lem:x[abc]iszero}. Since $a\in R^*$ is an arbitrary unit, we are done.
\end{proof}

\begin{proposition}
\label{prop:d2Surjects}
Let $R$ be a local ring with infinite residue field.
Then the map $\gamma: \Z[\Skew_6^+(R)] \to R \ [\Skew_3^+(R)]$ of Proposition \ref{prop:dAFormula} is surjective for $n=2$.
\end{proposition}

\begin{proof}
This is a reformulation of Lemma \ref{lem:x[abc]iszero}.
\end{proof}

\section{Localising homology groups}

The goal in this section is to show that if $\Sp_{2n+1}(R) \to \Sp_{2n+2}(R)$ is a surjection (injection, isomorphism) in homology then so is $\Sp_{2n}(R) \to \Sp_{2n+2}(R)$; see Proposition \ref{prop:HSpLocalization}.

Recall that 
$\Sp_{2n+1}(R)$ is the subgroup of $\Sp_{2n+2}(R)$  of matrices 
\begin{equation}
\label{eqn:Spodd2}
\left(\begin{smallmatrix}1& c & {^t\!u}\psi M  \\ 0 & 1& 0 \\  0 & u & M \end{smallmatrix}\right)
\end{equation}
where $\psi = \psi_{2n}$, $M\in \Sp_{2n}(R)$, $u\in R^{2n}$, $c\in R$.
We let the group $R^*$ of units of $R$ act from the left on $\Sp_{2n+2}(R)$ by conjugation with the matrix $T_b\in \Sp_{2n+2}(R)$ for $b\in R^*$ where 
$$T_b = \left(\begin{smallmatrix}b & 0 & 0 \\ 0 & b^{-1}  & 0 \\ 0 & 0 & 1_n\end{smallmatrix} \right)$$
and $1_n$ denotes the $n\times n$ identity matrix.
Note that $T_b\cdot A\cdot T_b^{-1}=A$ for $A\in \Sp_{2n}(R)$ and $T_b\cdot A\cdot T_b^{-1} \in \Sp_{2n+1}(R)$ for $A \in \Sp_{2n+1}(R)$ since
$$ T_b \cdot \left(\begin{smallmatrix}1& c & {^t\!u}\psi M  \\ 0 & 1& 0 \\  0 & u & M \end{smallmatrix}\right) \cdot T_b^{-1} = \left(\begin{smallmatrix}1& b^2c & b{^t\!u}\psi M  \\ 0 & 1& 0 \\  0 & bu & M \end{smallmatrix}\right).
$$
By functoriality, this defines an $R^*$-action (hence a left $\Z[R^*]$-module structure) on the homology groups $H_t(\Sp_q(R))$ for $q=2n, 2n+1, 2n+2$.
The action is trivial for $\Sp_{2n}$ and $\Sp_{2n+2}$, the latter because $T_b\in \Sp_{2n+2}(R)$, but that action is non-trivial for $\Sp_{2n+1}$, in general.
Let $m$ be an integer such that $m>2t$.
We choose units 
$a_1,...,a_m \in R^*$ such that for every non-empty subset $I \subset \{1,...,m\}$ the partial sum $a_I=\sum_{i\in I}a_i$ is a unit in $R$. 
This is possible since $R$ has infinite residue field. 
Let $s_m \in \Z[R^*]$ be the element 
$$s_m = - \sum_{\emptyset \neq I \subset \{1,...,m\}} (-1)^{|I|}\langle a_I\rangle \in \Z[R^*]$$
first considered in \cite[\S 2]{myEuler}
where $\langle u\rangle \in \Z[R^*]$ denotes the element of the group ring corresponding to $u\in R^*$.
Since $R^*$ acts trivially on $H_t(\Sp_{2n}(R))$ and $H_t(\Sp_{2n+2}(R))$, multiplication by $s_m$ on those groups is the identity map \cite[p. 7]{myEuler} in view of the equality
$$1 = - \sum_{\emptyset \neq I \subset \{1,...,m\}} (-1)^{|I|}.$$
In particular, we have
$$H_t(\Sp_{2n}(R)) = s_m^{-1}H_t(\Sp_{2n}(R)),\hspace{3ex}H_t(\Sp_{2n+2}(R)) = s_m^{-1}H_t(\Sp_{2n+2}(R)).$$

\begin{proposition}
\label{prop:HSpLocalization}
Let $m,n,t\geq 0$ be integers with $m>2t$, and let $R$ be a local ring with infinite residue field.
Then localisation of the maps $H_t(\Sp_{2n}(R)) \to H_t(\Sp_{2n+1}(R)) \to H_t(\Sp_{2n+2}(R))$ at $s_m\in \Z[R^*]$ induces a commutative diagram of abelian groups
$$\xymatrix{
H_t(\Sp_{2n}(R)) \ar[r] \ar[d]^{\cong} &  H_t(\Sp_{2n+1}(R)) \ar[r] \ar[d] & H_t(\Sp_{2n+2}(R)) \ar[d]^{\cong}\\
s_m^{-1}H_t(\Sp_{2n}(R)) \ar[r]_{\cong}  & s_m^{-1} H_t(\Sp_{2n+1}(R)) \ar[r]  & s_m^{-1}H_t(\Sp_{2n+2}(R))
}$$
in which the outer two vertical arrows and the lower left horizontal arrow are isomorphisms.
In particular, if  the map $H_t(\Sp_{2n+1}(R)) \to H_t(\Sp_{2n+2}(R))$ is a surjection (injection, isomorphism) then so is  $H_t(\Sp_{2n}(R)) \to H_t(\Sp_{2n+2}(R))$.
\end{proposition}

\begin{proof}
We have already seen that the two outer vertical maps are isomorphisms.
Also, a localisation of a surjection (injection, isomorphism) is a surjection (injection, isomorphism).
So, all we have to prove is that the lower left horizontal map is an isomorphism.
The inclusion of groups $\eps: \Sp_{2n}(R) \to \Sp_{2n+1}(R)$
has a retraction 
\begin{equation}
\label{eqn:retractRho}
\rho: \Sp_{2n+1}(R) \to \Sp_{2n}(R): 
\left(\begin{smallmatrix}1& c & {^t\!u}\psi M  \\ 0 & 1& 0 \\  0 & u & M \end{smallmatrix}\right)
\mapsto M.
\end{equation}
This defines an exact sequence of groups
$$1 \to G \longrightarrow \Sp_{2n+1}(R) \stackrel{\rho}{\longrightarrow} \Sp_{2n}(R) \to 1.$$
Since the map $\rho$ is $R^*$-equivariant, the group $R^*$ acts on the exact sequence and hence on the associated Hochschild-Serre spectral sequence
$$E^2_{p,q} = H_p(\Sp_{2n}(R), H_q(G)) \Rightarrow H_{p+q}(\Sp_{2n+1}(R)).$$ 
We localise that spectral sequence at $s_m\in \Z[R^*]$ to obtain the spectral sequence
\begin{equation}
\label{eqn:HochSerreLocal}
s_m^{-1} E^2_{p,q} =  H_p(\Sp_{2n}(R), s_m^{-1}H_q(G)) \Rightarrow s_m^{-1}H_{p+q}(\Sp_{2n+1}(R)).
\end{equation}
Here we used that $s_m^{-1}H_p(\Sp_{2n}(R), H_q(G))=  H_p(\Sp_{2n}(R), s_m^{-1}H_q(G))$
since $R^*$ acts trivially on $\Sp_{2n}(R)$.
There is a central extension of groups 
$$1 \to (R,+) \to G \to (R^{2n},+) \to 1$$
where the second map sends 
$$\left(\begin{smallmatrix}1& c & {^t\!u}\psi   \\ 0 & 1& 0 \\  0 & u & 1 \end{smallmatrix}\right) \in G$$
to $u\in R^{2n}$.
This central extension is $R^*$-equivariant where $b\in R^*$ acts on $(R,+)$ via multiplication by $b^2$ and on $(R^{2n},+)$ via multiplication by $b$.
We have $s_m^{-1}H_0(G)=\Z$ since $R^*$ acts trivially on $H_0(G)$.
By Proposition \cite[D.4]{myForm1}, we have $s_m^{-1}H_q(G)=0$ for $m>2q>0$.
In particular, the localised Hochschild-Serre spectral sequence (\ref{eqn:HochSerreLocal}) degenerates at $E^2$ for $m>2t \geq 2q$ to yield the isomorphism
$$\rho: s_m^{-1}H_t(\Sp_{2n+1}R)  \stackrel{\cong}{\longrightarrow} H_t(\Sp_{2n}R)$$
for $t<m/2$. Since $\rho$ is a retract of $\eps$, we are done.
\end{proof}

\section{Homology stability}

As always, $R$ will be a local ring with infinite residue field.
In this section we will use the spectral sequence (\ref{SpecSeq})
with differentials $d^r$ of bidegree $(r-1,-r)$ to deduce our homology stability results from the Introduction.
So far, we have proved the following.
\vspace{2ex}

\noindent{\bf Properties of the spectral sequence (\ref{SpecSeq})}.
\label{prop:SpecSeqProps}
\begin{enumerate}
\item
\label{spseq:item:1}
The abutment satisfies $H_{p+q}(\Sp_{2n},\Z [U_*]_{|\, *\leq 2n+1})=0$ for $p+q \leq 2n$; see Lemma \ref{lem:HSpC}. In particular, $E^{\infty}_{p,q}(2n)=0$ for $p+q\leq 2n$.
\item
\label{spseq:item:2}
$$E^1_{p,q}(2n)\cong H_p(\Sp_{2n-q})\otimes \Z[\Skew^+_{q}]$$
for $0 \leq q \leq 2n+1$ and any $p\in \Z$; see (\ref{eqn:E1ShapiroIdentn}).
\item
\label{spseq:item:3}
Under the isomorphism of (\ref{spseq:item:2}), the differential in $E^1(2n)$ is
$$d^1_{p,q} = \eps \otimes d: H_p(\Sp_{2n-q})\otimes \Z[\Skew^+_{q}] \to H_p(\Sp_{2n-q+1})\otimes \Z[\Skew^+_{q-1}]$$
for $1\leq q \leq 2n+1$ where $\eps$ is the map induced by inclusion of groups and $d$ is the differential of the complex $\Z[\Skew^+_*]$; see Lemma \ref{lem:commdiag}.
\item
\label{spseq:item:4}
The complex $(\Z[\Skew^+_*],d)$ is acyclic; see Lemma \ref{lem:SkewExact}.
In particular, $E^2_{0,q}(2n)=0$ for $0\leq q \leq 2n$.
\item
\label{spseq:item:5}
For $r\geq 2$ and even $q<2n$, the differentials in $E^r(2n)$ satisfy $d_{p,q}^r=0$; see Corollary \ref{cor:d1SurjIso}.
\item
\label{spseq:item:6}
For $n=2$, the composition 
$$E^2_{0,5}(4) \stackrel{d^2_{0,5}}{\longrightarrow} E^2_{1,3}(4) \subset E^1_{1,3}(4)$$
is surjective; see Proposition \ref{prop:d2Surjects}.
In particular, $d^2_{0,5}$ is surjective and $E^2_{1,3}(4) = E^1_{1,3}(4)$, that is, $d^1_{1,3}=0$ in the spectral sequence $E(4)$.
\end{enumerate}
\vspace{1ex}

\noindent{\bf $H_0$ stability}.
As is true for any group, we have 
$$H_0(\Sp_n(R))=\Z,\hspace{3ex} n\geq -1.$$ 
\vspace{1ex}

\noindent{\bf $H_1$ stability}.
It is known that the groups $\Sp_{2n}(R)$ are perfect ($R$ has infinite residue field), in particular, $H_1(\Sp_{2n}(R)) =0$ for $n\geq 0$.
The following theorem reproves this fact and extends it to the groups $\Sp_{2n+1}(R)$.

\begin{theorem}
\label{them:H1stability}
Let $R$ be a local ring with infinite residue field.
Then 
$$H_1\Sp_q(R) = \left\{\begin{array}{rl} R, & q=1\\ 0, & q\geq {-1}, q\neq 1.\end{array}\right.$$
\end{theorem}

\begin{proof}
The statement is clear for $q=-1,0,1$ since $\Sp_{-1}(R)=\Sp_0(R)=1$ and $\Sp_1(R)=R$.

For $q=2n\geq 2$, the group $E^1_{1,0}(2n)=H_1(\Sp_{2n})$ has no outgoing differential and only one incoming differential $d^1_{1,1}:H_1(\Sp_{2n-1}) \to H_1(\Sp_{2n})$ since $d^2_{0,2}=0$, by the Properties of the spectral sequence (\ref{SpecSeq}) item (\ref{spseq:item:5}) above. Since $E^{\infty}_{1,0}(2n)=0$, the map $d^1_{1,1}:H_1(\Sp_{2n-1}) \to H_1(\Sp_{2n})$ is surjective.
By Proposition \ref{prop:HSpLocalization}, $H_1(\Sp_{2n-2})\to H_1(\Sp_{2n})$ is then also surjective for $n\geq 1$. 
Since $0=H_1(\Sp_0)$, we find that $H_1(\Sp_{2n})=0$ for $n\geq 0$.

For $q=2n\geq 4$, that is, $n\geq 2$, we have $E^1_{1,1}(2n) = E^2_{1,1}(2n)$ since $E^1_{1,0}(2n) = H_1(\Sp_{2n})=0$.
The $d^1$-sequence $E^1_{0,4}(2n) \to E^1_{0,3}(2n) \to E^1_{0,2}(2n)$ is exact, by item (\ref{spseq:item:4}) above. 
Therefore, $E^2_{0,3}(2n)=0$ and $d^2_{0,3}=0$. 
Hence, $H_1(\Sp_{2n-1}) = E^1_{1,1}(2n) = E^2_{1,1}(2n) = E^{\infty}_{1,1}=0$, by item (\ref{spseq:item:1}).
\end{proof}

\noindent{\bf $H_2$ stability}.
Extending results of Matsumoto \cite{matsumoto}, van der Kallen \cite{vdK:K2} shows isomorphisms
$$ H_2(\Sp_2(R)) \stackrel{\cong}{\longrightarrow} H_2(\Sp_4(R)) \stackrel{\cong}{\longrightarrow} H_2(\Sp_6(R)) \stackrel{\cong}{\longrightarrow} H_2(\Sp_8(R)) \stackrel{\cong}{\longrightarrow} \cdots$$
for any local ring with infinite residue field $R$.
The following theorem reproves this fact and extends it to the groups $\Sp_{2n+1}(R)$.

\begin{theorem}
\label{them:H2stability}
Let $R$ be a local ring with infinite residue field.
Then for $q\geq 2$ inclusion of groups induces an isomorphism $H_2(\Sp_q(R)) \cong H_2(\Sp_{q+1}(R))$:
$$H_2(\Sp_2(R)) \stackrel{\cong}{\longrightarrow} H_2(\Sp_3(R)) \stackrel{\cong}{\longrightarrow} H_2(\Sp_4(R)) \stackrel{\cong}{\longrightarrow} H_2(\Sp_5(R)) \stackrel{\cong}{\longrightarrow} \cdots$$
\end{theorem}

\begin{proof}
For $2n\geq 4$, that is, $n\geq 2$, the term $E^1_{2,0}(2n)=H_2\Sp_{2n}$ has no outgoing differential and only one incoming differential $d^1:E^1_{2,1}(2n) \to E^1_{2,0}(2n)$ since $d^2_{1,2}=0$ (see Properties of the spectral sequence (\ref{SpecSeq}) item (\ref{spseq:item:5})), and $d^3_{0,3}=0$ (item (\ref{spseq:item:4})).
The term $E^1_{2,1}(2n)=H_2\Sp_{2n-1}$ has one outgoing differential $d^1:E^1_{2,1}(2n) \to E^1_{2,0}(2n)$ and no incoming differential because $d^1_{2,2}=0$ (item (\ref{spseq:item:5})), $d^2_{1,3}=0$ 
($E^1_{1,3}=H_1(\Sp_{2n-3})\otimes \Z[\Skew^+_3]=0$ for $n\geq 3$ and item (\ref{spseq:item:6}) for $n=2$)
and $d^3_{0,4}=0$ (because $E^2_{0,4}=0$, by item (\ref{spseq:item:4})).
Since $E^{\infty}_{2,0}(2n)=E^{\infty}_{2,1}(2n)=0$ the differential $d^1:E^1_{2,1}(2n) \to E^1_{2,0}(2n)$ is an isomorphism, that is, $H_2(\Sp_{2n-1}) \to H_2(\Sp_{2n})$ is an isomorphism.
By Proposition \ref{prop:HSpLocalization}, $H_2(\Sp_{2n-2})\to H_2(\Sp_{2n})$ is then also an isomorphism for $n\geq 2$. 
\end{proof}

\noindent{\bf $H_3$ stability}.
The following theorem improves results of Essert \cite[Theorem 3.9]{Essert},  Sprehn-Wahl \cite[Theorem A]{SprehnWahl} and answers a question of Hutchinson-Wendt \cite[Remark 9.6]{HutchinsonWendt}.

\begin{theorem}
\label{them:H3stability}
Let $R$ be a local ring with infinite residue field.
Then for $q\geq 4$ inclusion of groups induces an isomorphism $H_3(\Sp_q(R)) \cong H_3(\Sp_{q+1}(R))$ and surjections $H_3(\Sp_2(R)) \twoheadrightarrow H_3(\Sp_4(R))$ and $H_3(\Sp_3(R)) \twoheadrightarrow H_3(\Sp_4(R))$:
$$\xymatrix{
H_3(\Sp_2)\ \  \ar@{>->}[r] \ar@/_1pc/@{->>}[rr] &  H_3(\Sp_3)\ar@{->>}[r]  & H_3(\Sp_4) \ar[r]^{\cong} & H_3(\Sp_5) \ar[r]^{\cong} & H_3(\Sp_6) \ar[r]^{\cong} & \cdots
}$$
\end{theorem}

\begin{proof}
For $2n\geq 4$, the term $E^1_{3,0}=H_3(\Sp_{2n})$ has no outgoing differential and only one incoming differential $d^1_{3,1}:H_3(\Sp_{2n-1}) \to H_3(\Sp_{2n})$ since $d^2_{2,2}=0$ (Properties of the spectral sequence (\ref{SpecSeq}) item  (\ref{spseq:item:5})), $d^3_{1,3} = 0$ ($E^{3}_{1,3}(4)=0$, by item  (\ref{spseq:item:6}) for $2n=4$, and Theorem \ref{them:H1stability} for $2n > 4$), and $d^4_{0,4}=0$ (item  (\ref{spseq:item:4})).
Since $E^{\infty}_{3,0}=0$ (item  (\ref{spseq:item:1})), the stabilisation map $d^1_{3,1}:H_3(\Sp_{2n-1}) \to H_3(\Sp_{2n})$  is surjective.
By Proposition \ref{prop:HSpLocalization}, $H_3(\Sp_{2n-2}) \to H_3(\Sp_{2n})$ is then also surjective.

For $2n\geq 6$, that is, $n\geq 3$, the term $E^1_{3,1}=H_3(\Sp_{2n-1})$ has no outgoing differential other than $d^1_{3,1}$ and no incoming differential.
Indeed, $d^2_{2,3}=0$ as $E^2_{2,3}=0$ in view of Theorem \ref{them:H2stability} and items  (\ref{spseq:item:3}) and (\ref{spseq:item:4}).
The differential $d^3_{1,4}=0$ because $E^1_{1,4}=H_1(\Sp_{2n-4})\otimes \Z[\Skew^+_4]=0$, by Theorem \ref{them:H1stability}.
The differential $d^4_{0,5}=0$ because $E^2_{0,5}=0$, by item (\ref{spseq:item:4}).
Since $E^{\infty}_{3,1}=0$ (item  (\ref{spseq:item:1})), the stabilisation map $d^1_{3,1}:H_3(\Sp_{2n-1}) \to H_3(\Sp_{2n})$  is also injective, hence an isomorphism.
By Proposition \ref{prop:HSpLocalization}, $H_3(\Sp_{2n-2}) \to H_3(\Sp_{2n})$ is then also an isomorphism.
\end{proof}

\begin{remark}
\label{rmk:NotInjective}
The homology stability range in Theorem \ref{them:H3stability} is optimal, in general. 
Indeed, let $k$ be an infinite perfect field of characteristic not $2$ which is finitely generated over its prime field, then neither of the two surjective maps
\begin{equation}
\label{eqn:NotInjective}
H_3(\Sp_2(k)) \twoheadrightarrow H_3(\Sp_4(k)),\hspace{2ex}\text{and }\hspace{2ex} 
H_3(\Sp_3(k)) \twoheadrightarrow H_3(\Sp_4(k))
\end{equation}
is injective.
For the first map, this follows from \cite[Theorem 7.4]{HutchinsonWendt} since that map factors through $H_3(B\Sp_2(k[\Delta^{\bullet}]))$ in view of the isomorphisms
$$H_3(B\Sp_4(k)) \cong H_3(B\Sp(k)) \cong H_3(B\Sp(k[\Delta^{\bullet}]))$$
of Theorem \ref{them:H3stability} and homotopy invariance of symplectic $K$-theory for regular rings containing $1/2$.
Then the second map cannot be injective either, by Proposition \ref{prop:HSpLocalization}.
\end{remark}

An important consequence of Theorem \ref{them:H3stability}
is the following relative homology stability result.

\begin{theorem}\label{lem:apl}
	Let $R$ be a local ring with infinite residue field. Then inclusion of groups induces an isomorphism of relative integral homology groups for $i\leq 3$
	$$H_i(\SL_3(R),\Sp_2(R)) \stackrel{\cong}{\longrightarrow} H_i(\SL(R), \Sp(R)) \cong 
	\left\{
	\renewcommand\arraystretch{1.5}
	\begin{array}{cl} 0 & i \leq 2\\ K_3^{MW}(R) & i=3. \end{array}\right.$$
\end{theorem}

\begin{proof}
For $i\leq 3$ we have the following commutative diagram
	\[
	\xymatrix
	{
		H_i(\Sp_2) \ar@{->}[r]\ar@{->>}[d]^{\xi_0} & H_i(\SL_3) \ar@{->}[d]^{\xi_1}_{\cong}
		\ar@{->}[r]	& H_i(\SL_3,\Sp_2)	\ar@{->}[d]^{\xi_2}
		\ar@{->}[r]	& H_{i-1}(\Sp_2)  \ar@{->}[d]^{\xi_3}_{\cong} \ar@{->}[r]
		&H_{i-1}(\SL_3)\ar[d]^{\xi_4}_{\cong}		\\
		H_i(\Sp_4) \ar@{->}[r] & H_i(SL_4) 
		\ar@{->}[r]	& H_i(\SL_4,\Sp_4)	
		\ar@{->}[r]	& H_{i-1}(\Sp_4)   \ar@{->}[r]
		&H_{i-1}(\SL_4)	
	}
	\]
	where the rows are exact, the maps $\xi_1$ and $\xi_4$ are isomorphisms, by \cite[Theorem 5.37]{myEuler}, the map $\xi_0$ is surjective, and the map $\xi_3$ is an isomorphism, by Theorems \ref{them:H1stability}, \ref{them:H2stability} and \ref{them:H3stability}.
By the Five Lemma, $\xi_2$ is an isomorphism.
Similarly we have isomorphisms
$$H_i(\SL_4,\Sp_4)  \stackrel{\cong}{\longrightarrow} H_i(\SL_6,\Sp_6) \stackrel{\cong}{\longrightarrow} \cdots  \stackrel{\cong}{\longrightarrow}  H_i(\SL,\Sp)$$
	for $i\leq 3$. 
The identification with Milnor-Witt $K$-theory or $0$ follows from \cite[Theorem 5.37]{myEuler} using the identity $\Sp_2(R)=\SL_2(R)$.
\end{proof}

\noindent{\bf $H_n$ stability}.
The following result generalises Theorem \ref{them:H3stability}.
It is probably not optimal for $k\geq 1$ but it improves on \cite[Theorem 3.9]{Essert}, \cite[Theorem A]{SprehnWahl} .
In order to improve stability ranges further with the methods of this paper one would need to study various non-zero differentials in more detail which seems to be quite challenging.

\begin{theorem}
\label{them:Hkstability}
Let $R$ be a local ring with infinite residue field, and let $k\geq 0$ be an integer.
Then for $q\geq 2k+4$ inclusion of groups induces an isomorphism $H_{k+3}(\Sp_q(R)) \cong H_{k+3}(\Sp_{q+1}(R))$ and surjections\\
 $H_{k+3}(\Sp_{2k+2}(R)) \twoheadrightarrow H_{k+3}(\Sp_{2k+4}(R))$ and $H_{k+3}(\Sp_{2k+3}(R)) \twoheadrightarrow H_{k+3}(\Sp_{2k+4}(R))$:
{\tiny
$$\xymatrix{
H_{k+3}(\Sp_{2k+2})  \ar[d] \ar@{->>}[dr] & &&&\\
 H_{k+3}(\Sp_{2k+3})\ar@{->>}[r]  & H_{k+3}(\Sp_{2k+4}) \ar[r]^{\cong} & H_{k+3}(\Sp_{2k+5}) \ar[r]^{\cong} & H_{k+3}(\Sp_{2k+6}) \ar[r]^{\hspace{3ex}\cong} & \cdots
}$$
}
\end{theorem}

\begin{proof}
We prove the theorem by induction on $k$, the case $k=0$ is Theorem \ref{them:H3stability}.
Let $n\geq 1$ be an integer and assume the theorem is true for $k$ satisfying $0\leq k <n$.
We consider the Spectral sequence $E^r_{p,q}(2n+4)$ and show that all incoming differentials at $E^r_{n+3,0}(2n+4)$ are zero for $r\geq 2$.
Indeed, the incoming differential $d^r_{n+4-r,r}: E^r_{n+4-r,r}(2n+4) \to E^r_{n+3,0}(2n+4)$ is zero for $r\geq 2$ as $E^2_{n+4-r,r}(2n+4)=0$.
This is because it is the homology of the complex
$$\renewcommand\arraystretch{2}\begin{array}{rl}
H_{n+4-r}(\Sp_{2n+3-r})\otimes \Z[\Skew^+_{r+1}] & \stackrel{\eps\otimes d}{\longrightarrow}H_{n+4-r}(\Sp_{2n+4-r})\otimes \Z[\Skew^+_r]\\
 & \stackrel{\eps\otimes d}{\longrightarrow} H_{n+4-r}(\Sp_{2n+5-r})\otimes \Z[\Skew^+_{r-1}]
 \end{array}$$
which is exact, by induction hypothesis and the Properties of the spectral sequence (\ref{SpecSeq}) item (\ref{spseq:item:4}).
Since $E^{\infty}_{n+3,0}(2n+4)=0$, by item (\ref{spseq:item:1}), the stabilisation map
$d^1_{n+3,1}: H_{n+3}(\Sp_{2n+3}) \to H_{n+3}(\Sp_{2n+4})$ must be surjective.
By Proposition \ref{prop:HSpLocalization}, the stabilisation map 
$H_{n+3}(\Sp_{2n+2}) \to H_{n+3}(\Sp_{2n+4})$ is then also surjective.

For $m>n$ we consider the spectral sequence 
$E^r_{p,q}(2m+4)$ and show that all incoming differentials at $E^r_{n+3,0}(2m+4)$ and $E^r_{n+3,1}(2m+4)$ are zero for $r\geq 2$ (their outgoing differentials are zero anyway).
The incoming differentials $d^r_{n+4-r,r}: E^r_{n+4-r,r}(2m+4) \to E^r_{n+3,0}(2m+4)$ and 
$d^r_{n+4-r,r+1}: E^r_{n+4-r,r+1}(2m+4) \to E^r_{n+3,1}(2m+4)$
are zero for $r\geq 2$ as $E^2_{n+4-r,r}(2m+4)=0$ and $E^2_{n+4-r,r+1}(2m+4)=0$, 
by induction hypothesis and item (\ref{spseq:item:4}).
Since $E^{\infty}_{n+3,0}(2m+4)=E^{\infty}_{n+3,1}(2m+4)=0$, by item (\ref{spseq:item:1}), the stabilisation map
$d^1_{n+3,1}: H_{n+3}(\Sp_{2m+3}) \to H_{n+3}(\Sp_{2m+4})$ must be an isomorphism.
By Proposition \ref{prop:HSpLocalization}, the stabilisation map 
$H_{n+3}(\Sp_{2m+2}) \to H_{n+3}(\Sp_{2m+4})$ is then also an isomorphism.
\end{proof}

\section{The Hurewicz map}

Let $R$ be a local ring with infinite residue field. 
We denote by $K^{MW}_n(R)$ the $n$-th Milnor-Witt $K$-group of $R$ \cite[Definition 4.10]{myEuler}.
This is the evident generalisation to local rings of the definition for fields given in \cite{Morel:book}.
We recall the definition of symplectic $K$-theory $K_n\Sp(R) = \pi_nB\Sp(R)^+$ for $n\geq 1$ where $\Sp(R)=\bigcup_{n\geq 0}\Sp_{2n}(R)$ is the infinite symplectic group over $R$ and "+" is Quillen's plus construction with respect to the perfect (sub-) group $\Sp(R)$.

\begin{theorem}
\label{thm:HurewiczGW22}
Let $R$ be a local ring with infinite residue field.
Then the Hurewicz map induces an isomorphism
$$K_2\Sp(R) \stackrel{\cong}{\longrightarrow} H_2(\Sp(R)) \stackrel{\cong}{\longleftarrow}H_2(\Sp_2(R)) \cong K_2^{MW}(R).$$
\end{theorem}

\begin{proof}
The first map is the Hurewicz isomorphism since $\pi_1B\Sp(R)^+ = H_1\Sp(R) = 0$.
The second isomorphism follows from homology stability (Theorem \ref{them:H2stability}).
The last isomorphism is \cite[Theorem 5.27]{myEuler} since $\Sp_2(R)=\SL_2(R)$.
\end{proof}

We denote by $GW^{[3]}_3(R)$ the $2$-nd homotopy group of the homotopy fibre of the forgetful map
$K\Sp(R) \to K(R)$ induced by the inclusion of groups $\Sp(R) \subset \SL(R)$.
Thus, $GW^{[3]}_3(R)$ sits in an exact sequence
$$K_3\Sp(R) \to K_3(R) \to GW^{[3]}_3(R) \to K_2\Sp(R) \to K_2(R) \to \cdots$$
The definition given here agrees with that in \cite{myJPAA} when $\frac{1}{2}\in R$ in view of \cite[Theorem 6.1]{myJPAA}.

\begin{theorem}
\label{thm:HurewiczGW33}
	Let $R$ be a local ring with infinite residue field. 
	Then the Hurewicz map induces an isomorphism
	$$GW_3^{[3]}(R) \stackrel{\cong}{\longrightarrow} H_3(SL(R), \Sp(R)) \stackrel{\cong}{\longrightarrow} K_3^{MW}(R).$$
\end{theorem}
\begin{proof}
	By definition, we have $GW_3^{[3]}(R)\cong \pi_3(BSL^+(R),B\Sp^+(R))$. 
	It follows that
     $\pi_3(BSL^+,B\Sp^+)\cong H_3(SL, \Sp) \cong K^{MW}_3(R)$, by the relative Hurewicz theorem \cite[III Corollary 3.12]{GJ} and Theorem \ref{lem:apl},      	\end{proof}

\section{The $KO$-degree map}
\label{sec:KOdeg}

We write $\Z'$ for the ring $\Z[1/2]$.
In this section, all rings are commutative with trivial involution and have $2$ as a unit. 
Let $R$ be such a ring.
In \cite[Definition 9.1]{myJPAA} we have defined abelian groups $GW^{[n]}_i(R)$ as the homotopy groups $GW^{[n]}_i(R) =\pi_iGW^{[n]}(R)$, $i,n\in \Z$, $i\geq 0$, of a pointed topological space $GW^{[n]}(R)$ associated with the category of bounded chain complexes of finitely generated projective $R$-modules equipped with the duality $P \mapsto \Hom(P,R[n])$. 
For $n=0$, the space $GW^{[0]}(R)$ is the $K$-theory space of non-degenerate symmetric bilinear forms over $R$, and for $n=2$, the space $GW^{[2]}(R)$ is the $K$-theory space of non-degenerate symplectic forms over $R$.
In particular, the groups $GW^{[0]}_0(R)$ and $GW^{[2]}_0(R)$ are the Grothendieck groups of non-degenerate symmetric and symplectic forms over $R$.
Using a different definition, the groups $GW^{[r]}_m(R)$ have been introduced by Karoubi \cite{Karoubi:batelle} under the name Hermitian $K$-theory where $GW^{[1]}(R)$ and $GW^{[-1]}(R)$ are equivalent to Karoubi's spaces $_{-1}U(R)$ and $_1U(R)$.

There are natural, associative and unital cup product maps \cite[\S 9.2]{myJPAA}
$$\cup: GW^{[r]}_m(R) \otimes GW^{[s]}_n(R) \to GW^{[r+s]}_{m+n}(R)$$
making $\bigoplus_{s,n}GW^{[s]}_n(S)$ into a bi-graded $\bigoplus_{r,m}GW^{[r]}_m(R)$-algebra for any $R$-algebra $S$.
Moreover, there is a natural Bott long exact sequence  \cite[Theorem 6.1]{myJPAA}
$$
\cdots \to GW^{[r]}_m(R)\stackrel{f}{\longrightarrow} K_m(R) \stackrel{h}{\longrightarrow} GW^{[r+1]}_m(R) \stackrel{\eta}{\longrightarrow} GW^{r}_{m-1}(R) \stackrel{f}{\longrightarrow} K_{m-1}(R) \to \cdots$$
where $f$ and $h$ are forgetful and hyperbolic maps and $\eta$ is cup product with $\eta \in GW^{[-1]}_{-1}(\Z') \cong W(\Z')$ corresponding to the identity element in the Witt ring $W(\Z')$.
\vspace{2ex}

The Bott sequence \cite[Theorem 6.1]{myJPAA} in low degrees gives an exact sequence
\begin{equation}
\label{eqn:BottLowDg}
GW^{[0]}_1(R)\stackrel{f}{\longrightarrow} K_1(R) \stackrel{h}{\longrightarrow} GW^{[1]}_1(R) \stackrel{\eta}{\longrightarrow} GW^{[0]}_0(R) \stackrel{f}{\longrightarrow} K_0(R).
\end{equation}
Assume that $SK_1(R)=0$, that is, the determinant map $\det:K_1(R) \to R^*$ is an isomorphism.
Then the composition $K_1(R) \stackrel{h}{\to} GW^{[1]}_1(R) \stackrel{f}{\to} K_1(R)$ is multiplication by $2$, and we have a map of exact sequences
\begin{equation}
\label{eqn:MapOfexSeq1}
\xymatrix{
K_1(R) \ar[r]^h \ar[d]_1 & GW^{[1]}_1(R) \ar@{->>}[r]^{\eta} \ar[d]^f & I(R) \ar[d] \\
K_1(R) \ar[r]_2 & K_1(R) \ar@{->>}[r] & k_1(R)
}
\end{equation}
where $k_1(R) = K_1(R)/2$, $I(R)=\ker(f: GW^{[0]}_0(R) \to K_0(R))$ and the right vertical map is the map on cokernels of the two left horizontal maps.

\begin{lemma}
\label{lem:GW11toFibProd}
Assume that $SK_1(R) =0$.
Then the following map is an isomorphism:
\begin{equation}
\label{eqn:GW11toFibProd}
(f,\eta): GW^{[1]}_1(R) \stackrel{\cong}{\longrightarrow} K_1(R)\times_{k_1(R)}I(R).
\end{equation}
\end{lemma}

\begin{proof}
The map is surjective in view of the commutative diagram with exact rows (\ref{eqn:MapOfexSeq1}).
For injectivity, let $\xi\in GW^{[1]}_1(R)$ such that $\eta(\xi) =0$ and $f(\xi)=0$.
Then there is $a\in R^*=K_1(R)$ such that $h(a)=\xi$, by the exact sequence (\ref{eqn:BottLowDg}).
Then $a^2=fh(a)=f(\xi)=1 \in R^*$, and $a\in K_1(R)$ is in the image of the forgetful map $GW^{[0]}_1 (R) \to K_1(R)$, namely, it is the image of the automorphism the unit bilinear space $\langle 1 \rangle$ given by multiplication with $a$.
By exactness of (\ref{eqn:BottLowDg}), $\xi=h(a)=0$.
This proves the injectivity claim.
\end{proof}

Let $R$ be a commutative $\Z'$-algebra, and $R^*$ its group of units.
The integral group ring $\Z[R^*]$ has $\Z$-basis the elements $\langle a\rangle$, $a\in R^*$, and multiplication $\langle a\rangle \cdot \langle b\rangle = \langle ab\rangle$.
If we also denote by $\langle a \rangle$ the rank $1$ bilinear space $R$ equipped with the form $(x,y) \mapsto axy$, then the abelian group homomorphism
$$\Z[R^*] \to GW_0^0(R): \langle a \rangle \mapsto \langle a \rangle$$
is a ring homomorphism.
This makes all groups $GW^{[r]}_n(R)$ into $R^*$-modules.
 
The augmentation ideal $I[R^*]$ is the kernel of the ring homomorphism $\Z[R^*] \to \Z: \langle a\rangle \mapsto 1$.
We write $\langle\langle a \rangle\rangle$ for the element $1-\langle a \rangle$ considered in $\Z[R^*]$, and we write $[a]$ for the same element $1-\langle a \rangle$ considered in $I[R^*]$.
Thus, the inclusion $I[R^*] \to \Z[R^*]$ sends $[a]$ to 
$\langle\langle a\rangle\rangle$\footnote{In \cite{myEuler}, we used the convention $[a]=\langle a\rangle -1$.}.
The augmentation ideal has $\Z$-basis $[a]$, $a\in R^*$, $a\neq 1$.

The Bass Fundamental Theorem for Grothendieck-Witt groups \cite[Theorem 9.13]{myJPAA} together with homotopy invariance \cite[Theorem 9.8]{myJPAA} for regular rings gives a split short exact sequence
$$0 \to GW^{[1]}_1(\Z') \longrightarrow GW^{[1]}_1(\Z'[T,T^{-1}]) \stackrel{\delta}{\longrightarrow} GW^{[0]}_0(\Z') \longrightarrow 0.$$
The left map is split by the map $\eps$ evaluating at $T=1$.
As in \cite[Paragraph before Lemma 6.9]{KSW}, we denote by $[T] \in GW^{[1]}_1(\Z'[T,T^{-1}])$ the unique element such that $\delta([T]) = 1 \in GW^{[0]}_0(\Z')$ and $\eps([T])=0 \in GW^{[1]}_1(\Z')$.
For $a\in R^*$, denote by $[a]\in GW^{[1]}_1(R)$ the image of $[T]$ under the map 
$GW^{[1]}_1(\Z'[T,T^{-1}]) \to GW^{[1]}_1(R)$ induced by $\Z'[T,T^{-1}]\to R: T \mapsto a$.
Note that $[1]$ is $0$ in $GW^{[1]}_1(R)$ since $\eps([T])=1$.
We define the abelian group homomorphism
$$\ffi: I[R^*] \longrightarrow GW^{[1]}_1(R): [a] \mapsto [a].$$

\begin{lemma}
For any commutative $\Z'$-algebra, the following diagram commutes and is $R^*$-equivariant
$$
\xymatrix{
I[R^*] \ar[rr]^{[a] \mapsto [a]} \ar@{^(->}[d] && GW^{[1]}_1(R) \ar[d]^{\eta}\\
\Z[R^*] \ar[rr]_{\langle a \rangle \mapsto \langle a \rangle} && GW^{[0]}_0(R)
}$$
\end{lemma}

\begin{proof}
For $R=\Z'[T,T^{-1}]$ we have $\eta([T]) = \langle T \rangle$, by \cite[Lemma 6.9]{KSW}.
For general $R$ and $a\in R^*$, let $\eps_a:\Z'[T,T^{-1}] \to R$ be the $\Z'$-algebra homomorphism sending $T$ to $a$.
Then $\eta([a]) = \eta([\eps_a(T)]) = \eta(\eps_a([T])) = \eps_a \eta([T]) = \eps_a (\langle T \rangle) = \langle a\rangle$, and the diagram commutes.

The map $\eta$ is a $GW^{[0]}_0(R)$-module map, in particular, it is $R^*$-equivariant.
The lower horizontal and the left vertical maps are $R^*$-equivariant by definition.
For the top horizontal map we only need to consider products $\langle b\rangle \cdot [a]$ with $a,b\in R^*$.
In other words, it suffices to show equivariance for $\Z'[S,T,S^{-1},T^{-1}]$, then use functoriality under the ring homomorphism $\Z'[S,T,S^{-1},T^{-1}] \to R: S \mapsto a,\ T \mapsto b$.
Recall that $SK_1(\Z'[S,T,S^{-1},T^{-1}])=0$.
In particular, we only need to check equivariance of the top map for rings with trivial $SK_1$.
In this case, we can use the $R^*$-equivariant isomorphism (\ref{eqn:GW11toFibProd}) and have to show that the composition
$$(\ffi_1,\ffi_2): I[R^*] \longrightarrow GW^{[1]}_1(R) \stackrel{(f,\eta)}{\longrightarrow} K_1(R)\times_{k_1(R)}I(R)$$
is equivariant.
Now $\ffi_1([a]) = a \in K_1(R)$ (as $f([T]) = T \in K_1(\Z'[T,T^{-1}])$) and $\ffi_2([a]) = 1-\langle a\rangle$ (since the diagram commutes).
In particular, $\ffi_2$ is $R^*$-eqiuvariant.
Since the $R^*$-action on $K_n(R)$ is via the forgetful map $GW^0_0(R) \to K_0(R)$, that action is trivial, and we have $\ffi_1(\langle b\rangle \cdot [a]) = \ffi_1([ba] - [b]) = bab^{-1}=a = \langle b\rangle \cdot \ffi_1([a])\in K_1(R)$.
\end{proof}

For a connected ring $R$, let $SK_0(R)$ be the kernel of the rank homomorphism $K_0(R) \to \Z$.

\begin{lemma}
\label{lem:GW11Square}
Let $R$ be a connected commutative ring with $SK_i(R)=0$ for $i=0,1$.
Then the right square of (\ref{eqn:MapOfexSeq1}) is cartesian:
\begin{equation}
\label{eqn:GW1Square}
\xymatrix{
 GW^{[1]}_1(R) \ar@{->>}[r]^{\eta} \ar@{->>}[d]_f & I(R) \ar@{->>}[d] \\
 K_1(R) \ar@{->>}[r] & k_1(R).
}
\end{equation}
The right vertical map $I(R) \to k_1(R)$ sends $[a]$ to $\left\{ a\right\}$.
If $R$ is local then the vertical maps are surjections and the map
$I[R^*] \to GW^{[1]}_1(R):[a] \mapsto [a]$ is also surjective.
\end{lemma}

\begin{proof}
The diagram is cartesian, by Lemma \ref{lem:GW11toFibProd}.
The composition 
$$I[R^*] \to GW^{[1]}_1(R) \stackrel{f}{\to} K_1(R) \to k_1(R)$$
 sends $[a]$ to $ a \in k_1(R)$ since $f([T]) = T \in K_1(\Z'[T,T^{-1}])$.
Therefore,  the composition $I[R^*] \to I(R)  \to k_1(R)$ is $[a] \mapsto [a] \mapsto a$.

If $R$ is local then the map $f$ is surjective because $GW^{[1]}_1(R) \stackrel{f}{\to} K_1(R) \stackrel{h}{\to} GW^{[2]}_1(R)$ is exact and $GW^{[2]}_1(R)=K_1\Sp(R) \cong H_1(\Sp(R))=0$ for every local ring. 
Then the right vertical map is also surjective.
Finally, the map $I[R^*] \to GW^{[1]}_1(R)$ is surjective since the composition with the isomorphism 
(\ref{eqn:GW11toFibProd}) is surjective.
This follows from the commutative diagram with exact 2nd row and outer vertical maps surjective:
$$\xymatrix{
0 \ar[r] & I[R^{2*}] \ar@{->>}[d] \ar[r] & I[R^*] \ar[d] \ar@{->>}[r] & I[R^*/R^{2*}] \ar@{->>}[d]  \ar[r] & 0\\
0 \ar[r] & 2K_1(R) \ar[r] & K_1(R)\times_{k_1(R)}I(R) \ar[r] & I(R) \ar[r] & 0}
$$
\end{proof}

\begin{lemma}[Steinberg Relation]
\label{lem:SteinbergRln}
For any commutative ring $R$ over $\Z'$ and any $a\in R^*$ with $1-a\in R^*$, we have 
$$[a]\cup [1-a]=0\in GW^{[2]}_2(R).$$
\end{lemma}

\begin{proof}
The proof given in \cite[Proposition 4.1.10]{AsokFasel:KODegree} goes through.
In detail, consider the ring $S=R[T]/(T^2-a)$.
Then $T$ and $1-T$ are units in $S$, the latter because the matrix in  the $R$-basis $1,T$ of $S$ has determinant $1-a \in R^*$.
Consider the $R$-linear map $i: S \to R$ sending $1$ to $1$ and $T$ to $0$ (up to a factor of $2$, this is the trace used in \cite{AsokFasel:KODegree}).
It sends a symmetric bilinear form $(P,b)$ over $S$ to the symmetric bilinear form $i_*(P,b) =(P, i \circ b)$ over $R$ and preserves non-degeneracy since the unit form $\langle 1 \rangle$ is sent to the diagonal form $\langle 1\rangle + \langle a \rangle$.
Therefore, $i_*$ defines a map of Grothendieck-Witt groups $GW_n^{[r]} (S) \to GW_n^{[r]}(R)$ which is $GW^{[*]}_*(R)$-linear.
The Gram matrix of $i_*\langle 1-T\rangle$ with respect to the $R$-basis $1,T$ is $\left(\begin{smallmatrix} 1 & -a \\ -a &a \end{smallmatrix}\right)$ which has determinant $a - a^2$ and contains the non-degenerate subspace $\langle a \rangle $ with basis $T$.
Therefore, the orthogonal subspace is $\langle 1-a\rangle $ and we have $i_*\langle 1-T\rangle = \langle a \rangle + \langle 1-a\rangle\in GW_{[0]}^0(R)$.
It follows that $i_* [1-T] = [1-a] \in GW^{[1]}_1(R)$, first by considering $R=\Z'[t,t^{-1},(1-t)^{-1}]$ using Lemma \ref{lem:GW11Square}, then using functoriality and the map $\Z'[t,t^{-1},(1-t)^{-1}] \to R: t \mapsto a$.
By the Steinberg relation in $K$-theory, we have $0 = \{ T\} \cup \{1-T\} = f([T]\cup [1-T] ) \in K_2(S)$ using that the forgetful map $f:GW^{[*]}_*(S) \to K_*(S)$ is a map of graded rings.
Then 
$$0=hf([T]\cup [1-T] ) = \langle 1,-1\rangle \cdot [T]  \cup [1-T] = [T^2] \cup [1-T] = [a] \cup [1-T] \in GW^{[2]}_2(S)$$
since $ \langle 1,-1\rangle \cdot [T]  = [T^2] =[a] \in GW^{[1]}_1(S)$ (it suffices to check the equation $\langle 1, -1 \rangle \cdot [T] =[T^2]$ in $GW^{[1]}_1(\Z'[T,T^{-1}])$ using Lemma \ref{lem:GW11toFibProd}).
Applying the transfer map $i_*$ and noting that it is $GW^*_*(R)$-linear, we obtain
$$0 = i_*( [a] \cup [1-T] ) = [a] \cup i_*[1-T] = [a]\cup [1-a] \in GW^{[2]}_2(R).$$
\end{proof}

For a commutative $\Z'$-algebra denote by $\widehat{K}_*^{MW}(R)$ and $K_*^{MW}(R)$ the graded $\Z[R^*]$ and $GW_0^0(R)$-algebras 
{\small
$$\widehat{K}_*^{MW}(R) = \Tens_{\Z[R^*]} I[R^*] / [a][1-a],\hspace{3ex}K_*^{MW}(R) = \Tens_{GW^0_0(R)}GW^{1}_1(R) /[a][1-a]$$
}
where the relations run over all $a, 1-a\in R^*$.
By Lemma \ref{lem:SteinbergRln}, the maps $\Z[R^*] \to GW_0^0(R)$ and $I[R^*] \to GW^{[1]}_1(R)$ extend uniquely to homomorphisms of graded rings, called {\em symbol map} or {\em $KO$-degree map},
\begin{equation}
\label{eqn:KOmap}
\widehat{K}^{MW}_*(R) \longrightarrow K^{MW}_*(R)  \longrightarrow \bigoplus_{n\geq 0} GW^{[n]}_n(R).
\end{equation}

When $R$ is local with infinite residue field, the first map is an isomorphism in degrees $*\geq 2$ \cite[Theorem 4.18]{myEuler}.

In the following we will use the main result from \cite{GilleEtAl} which states that if $R$ is a local ring containing an infinite field of characteristic not two, then the following square is cartesian
\begin{equation}
\label{eqnLGilleEtAl}
\xymatrix{
K^{MW}_n(R) \ar@{->>}[r]^{\eta^n}  \ar@{->>}[d]^f &  I^n(R)  \ar[d] \\
K^M_n(R) \ar[r] & k^M_2(R)}
\end{equation}
where the right vertical map sends $[a_1] \cup \cdots \cup [a_n]$ to $\{a_1\} \cup \cdots \{a_n\}$ and induces an isomorphism $I^n(R)/I^{n+1}(R) \cong k_2^M(R)$.
In other words, the map
\begin{equation}
\label{eqn:LGilleEtAl_ISO}
\lambda: K^{MW}_n(R) \stackrel{\cong}{\longrightarrow} K^M_n(R) \times_{k_n}I^n(R): [a_1,...,a_n]\mapsto \left\{a_1,...,a_n\right\}, [a_1]\cdots [a_n]
\end{equation}
is an isomorphism.
In addition to the forgetful map 
$$f:K^{MW}_n(R) \to K^M_n(R):[a_1] \cup \cdots \cup [a_n] \mapsto \{a_1\} \cup \cdots \{a_n\}$$ we also have the hyperbolic map 
$$h:K^M_n(R) \to K^{MW}_n(R):\{a_1\} \cup \cdots \{a_n\} \mapsto \langle 1,-1\rangle [a_1] \cup \cdots \cup [a_n]$$ and the map 
$$\eta:K^{MW}_{n}(R) \to K^{MW}_{n-1}(R): [a_1] \cup[a_2] \cup \cdots \cup [a_n] \mapsto (1-\langle a_1\rangle) \cdot [a_2] \cup \cdots \cup [a_n].$$
Under the isomorphism (\ref{eqn:LGilleEtAl_ISO}), forgetful, hyperbolic map and $\eta$ are 
$K_n(R)\times_{k_n}I^n(R) \to K_n(R):(x,y) \mapsto x$, $K_n(R) \to K_n(R)\times_{k_n}I^n(R): x \mapsto (2x,0)$ and $K_n(R)\times_{k_n}I^n(R) \to K_{n-1}(R)\times_{k_{n-1}}I^{n-1}(R):(x,y) \mapsto (0,y)$.
In particular, for all integers $n\geq 1$, the following sequence is exact
$$K_n^M(R) \stackrel{h}{\longrightarrow}  K_n^{MW}(R) \stackrel{\eta}{\longrightarrow} K_{n-1}^{MW}(R) \stackrel{f}{\longrightarrow} K_{n-1}^M(R) \to 0$$
as it corresponds under the isomorphism (\ref{eqn:LGilleEtAl_ISO}) to the exact sequence
$$K_n^M \stackrel{(2,0)}{\longrightarrow}  K_n^{M}\times_{k_n}I^n \stackrel{(0\times 1)}{\longrightarrow} K_{n-1}^{M}\times_{k_{n-1}}I^{n-1} \stackrel{(1,0)}{\longrightarrow} K_{n-1}^M \to 0.$$
For the next theorem recall that $GW^{[2]}_2(R) = \pi_2B\Sp(R)^+$.

\begin{theorem}
\label{thm:GW22}
Let $R$ be a local commutative ring containing an infinite field of characteristic $\car(k) \neq 2$.
Then the $KO$-degree map is an isomorphism in degree $2$:
$$K_2^{MW}(R) \stackrel{\cong}{\longrightarrow} GW^{[2]}_2(R).$$
\end{theorem}

\begin{proof}
The map in the theorem is injective because the following composition is injective
$$K_2^{MW}(R) \to GW^{[2]}_2(R) \stackrel{(f,\eta)}{\longrightarrow} K_2(R) \times GW_1^{[1]}(R)$$
as that composition can be identified with the injective map 
$$\left(\begin{smallmatrix}1 & 0\\ 0 & 0 \\ 0 & i\end{smallmatrix}\right): K_2^M(R)\times_{k_2}I^2(R) \to K_2(R) \times K_1(R)\times_{k_1}I(R)$$
where $i:I^2(R) \subset I(R)$ is the inclusion.
The map in the theorem is surjective since it is part of the following map of exact sequences
$$
\xymatrix{
K_2^M(R) \ar[d]^{\cong} \ar[r]^h & K_2^{MW}(R) \ar[d] \ar[r]^{\eta} & K_1^{MW}(R) \ar[d]^{\cong} \ar[r]^f & K_1^M(R) \ar[d]^{\cong}\\
K_2(R) \ar[r]^h & GW_2^{[2]}(R)  \ar[r]^{\eta} & GW_1^{[1]}(R) \ar[r]^f & K_1(R).
}$$
\end{proof}

Consider the diagram (suppressing the local ring $R$ in the notation)
\begin{equation}
\label{eqn:refinedHurewicz}
\xymatrix{
K^M_3 \ar[r] \ar[d]_h & K_3 \ar[r] \ar[d]_h \ar@{}[dr]|-{(\ast)}  & H_3(SL) \ar[d] & H_3(SL_3) \ar[d] \ar[l]_{\cong} \ar[r] & K^{MW}_3 \ar[d]^1\\
K^{MW}_3 \ar[r] \ar[d]_{\eta} & GW_3^{[3]}  \ar@{}[dr]|-{(\ast)} \ar[d]_{\eta} \ar[r]^{\cong} & H_3(\SL,\Sp) \ar[d] & H_3(SL_3,SL_2) \ar@{=}[r] \ar[d] \ar[l]_{\cong} \ar@{}[dr]|-{(\ast\ast)}& K_3^{MW} \ar[d]^{\eta}\\
K^{MW}_2 \ar[r]_{\cong} \ar@{->>}[d]_f & GW^{[2]}_2 \ar[d]_f \ar@{}[dr]|-{(\ast)} \ar[r]_{\cong} & H_2(\Sp) \ar[d] & H_2(SL_2)  \ar@{=}[r] \ar@{->>}[d]\ar[l]^{\cong} \ar@{=} & K_2^{MW} \ar@{->>}[d]\\
K^M_2 \ar[r]_{\cong} & K_2 \ar[r]_{\cong} & H_2(SL) & H_2(SL_3) \ar[l]^{\cong} \ar[r]_{\cong} & K^M_2}
\end{equation}
where the horizontal arrows in the first column are the degree maps, in the second column they are the Hurewicz maps (with second and third row isomorphisms, by Theorems \ref{thm:HurewiczGW22} and  \ref{thm:HurewiczGW33}).
In the third column, the arrows in the wrong direction are the group homology stability isomorphisms of Theorems \ref{lem:apl} and \ref{them:H2stability} for the middle two and \cite[Theorem 5.37]{myEuler} for the top and bottom map.

\begin{lemma}
Diagram (\ref{eqn:refinedHurewicz}) commutes.
\end{lemma}

\begin{proof}
Diagrams $(\ast)$ commute in view of the compatibility of the Hurewicz map with the long exact sequence of homotopy and homology groups for pairs (which, for simplicial sets, is just functoriality of the long exact sequence of homotopy groups for pairs applied to the natural transformation $X \to \tilde{\Z}[X]=\Z[X]/\Z x_0:x\mapsto x$; see \cite[Section III.3]{GJ}).
The square $(\ast\ast)$ commutes, by \cite[Lemma 5.40(2)]{myEuler}.
The rest follows by functoriality and the definition of the $KO$-degree map.
\end{proof}

\begin{lemma}
\label{KMtoKMWisH}
The composition $\ffi: K^M_3 \to K^{MW}_3$ in the string of maps
$$K^M_3 \longrightarrow  K_3  \longrightarrow H_3(SL) \stackrel{\cong}{\longleftarrow}   H_3(SL_3) \longrightarrow  K^{MW}_3$$
that is the top row of diagram (\ref{eqn:refinedHurewicz}) is the hyperbolic map $h:K_3^M \to K_3^{MW}$.
\end{lemma}

\begin{proof}
Under the isomorphism (\ref{eqn:LGilleEtAl_ISO}),
the map $h:K^M_n(R) \to K^{MW}_n(R)$ becomes $(2,0): K^M_n(R) \times_{k_n}I^n$, and the map $\eta:K^{MW}_n(R) \to K^{MW}_{n-1}(R)$ becomes $(0,i)$ where $i:I^n \subset I^{n-1}$ is the inclusion.
By \cite[Theorem 4.1(a)]{SuslinNesterenko}, the first factor $K^M_3(R) \to K^M_3(R)$ of $\lambda\ffi$ is multiplication by $2$.
The second factor $K^M_3(R) \to I^3$ of $\lambda\ffi$ is zero since the following subdiagram of (\ref{eqn:refinedHurewicz}) commutes
$$\xymatrix{
 K_3 \ar[r] \ar[d]^h  & H_3(SL) \ar[d] & H_3(SL_3) \ar[d] \ar[l]_{\cong} \ar[r] & K^{MW}_3 \ar[d]^1 \ar[r] & I^3 \ar[d]^1\\
GW^3_3   \ar[d]_{\eta} \ar[r]^{\cong} & H_3(\SL,\Sp) \ar[d] & H_3(SL_3,SL_2) \ar@{=}[r] \ar[d] \ar[l]_{\cong}& K_3^{MW} \ar[d]^{\eta}\ar[r] & I^3 \ar@{^(->}[d]^i\\
GW^2_2 \ar[r]_{\cong} & H_2(\Sp) & H_2(SL_2)  \ar@{=}[r] \ar[l]^{\cong}& K_2^{MW} \ar[r] & I^2,
}
$$
the composition of the left vertical maps is zero whereas the composition of the right vertical maps is injective.
\end{proof}

\begin{theorem}
\label{thm:GW33}
Let  $R$ be a local commutative ring containing an infinite field of characteristic $\car(k) \neq 2$.
Then the $KO$-degree map is an isomorphism in degree $3$:
$$K_3^{MW}(R) \stackrel{\cong}{\longrightarrow} GW^{[3]}_3(R).$$
\end{theorem}

\begin{proof}
We only need to show that the composition $\phi: K^{MW}_3 \to K^{MW}_3$ of the second top row in  diagram (\ref{eqn:refinedHurewicz}) is an isomorphism.
By Lemma \ref{KMtoKMWisH}, that map is part of the commutative diagram with exact rows
$$\xymatrix{
K^M_3 \ar[r]^h \ar[d]^1 & K^{MW}_3 \ar[r]^{\eta} \ar[d]^{\phi} & K^{MW}_2 \ar@{->>}[r]^f \ar[d]^{\cong} & K^M_2 \ar[d]^{\cong}\\
K^M_3 \ar[r]^h  & K^{MW}_3 \ar[r]^{\eta} & K^{MW}_2 \ar@{->>}[r]^f & K^M_2.
}$$
The map on the left horizontal kernels is the identity map.
Hence, by the Five Lemma, $\phi$ is an isomorphism.
\end{proof}

\begin{lemma}
\label{lem:gwvanish}
	Let $R$ be a local ring with infinite residue field.
	Then $GW_2^{[3]}(R)=0.$
\end{lemma}
\begin{proof}
By definition, we have $GW_2^3(R)=\pi_2(BSL^+(R),B\Sp^+(R))$. 
By the relative Hurewicz Theorem \cite[III Corollary 3.12]{GJ} that group is
$H_2(SL(R),\Sp(R))$.
 By Theorem \ref{lem:apl}, this is
 $H_2(SL_3(R), SL_2(R))$ which is
 zero, by \cite[Theorem 5.37]{myEuler}.
\end{proof}

We recall that $K_3^{ind}(R)=\coker(K_3^M(R) \rightarrow K_3(R))$.

\begin{corollary}
\label{cor:Kind}
	Let $R$ be a local ring containing an infinite field of characteristic not $2$. 
Then
$K_3^{ind}(R) \cong KO_3(R)$.
\end{corollary}
\begin{proof}
Consider the commutative diagram
$$
	\xymatrix{
		K_3^{MW}(R) \ar@{->>}[r] \ar@{->}[d]_{\cong}
		& K_3^M(R) \ar@{->}[d] &&\\
		GW_3^{[3]}(R)   \ar@{->}[r]& K_3(R)  \ar[r] & GW^{[4]}_3(R) \ar[r] & GW^{[3]}_2(R) =0
	}
$$
in which the horizontal arrows in the square are the forgetful maps (in particular the top map is surjective), the vertical maps are the $KO$ and $K$-theory degree maps, and the bottom row is the Bott exact sequence. 
The group $GW^{[3]}_2(R)$ is zero, by Lemma \ref{lem:gwvanish}, and the left vertical map is an isomorphism, by Theorem \ref{thm:GW33}.
Since $GW_3^{[4]}(R)=KO_3(R)$, the result follows.
\end{proof}

\appendix

\section{The Pfaffian of some matrices}

For space reasons, we will write $[a,\ b,\ c]$ for the skew-symmetric matrix
$$[a,\ b,\ c]= \left(
\begin{matrix}
0 &  a &  b  \\
-a  &  0 & c \\
-b & -c &   0
\end{matrix}
\right).$$

\newpage

\subsection{The skew-symmetric matrix of Lemma \ref{lem:dAexample1aRewrite}}
\label{App:lem:dAexample1aRewrite}
The $6\times 6$ skew-symmetric matrix 
$$A = \left(
\begin{matrix}
0 &  a &  a & a & a & a \\
-a  &  0 & b & b & b & b\\
-a & -b &   0 & c & c & c\\
-a & -b & -c &  0 & d & d\\
-a &-b &-c &-d & 0 & e\\
-a &-b &-c &-d &-e & 0
\end{matrix}
\right)$$
has Pfaffian 
$$\Pf(A) = ace,$$
and for $1\leq i<j<k\leq 6$, the values of $(-1)^{i+j+k}\frac{\Pf(A)}{\Pf(A^{\wedge}_{ij})\Pf(A^{\wedge}_{ik})
\Pf(A^{\wedge}_{jk})}$ and $\left[A_{\widehat{ijk}}\right]$ are as follows.
$$
\renewcommand\arraystretch{1.5}
\begin{array}{c|c|c}
(i,j,k) & \displaystyle{(-1)^{i+j+k}\frac{\Pf(A)}{\Pf(A^{\wedge}_{ij})\Pf(A^{\wedge}_{ik})
\Pf(A^{\wedge}_{jk})}} &  \left[A_{\widehat{ijk}}\right]\\
&\\
\hline
(1, 2, 3) &1/(be^2)& [d, d, e]\\
(1, 2, 4) &(-1)/(be^2)& [c, c, e]\\
(1, 3, 4) &c/(b^2e^2) &[b, b, e]\\
(2, 3, 4) &(-c)/(a^2e^2)& [a, a, e]\\
(1, 2, 5) &1/(bd^2) &[c, c, d]\\
(1, 3, 5) &(-c)/(b^2d^2) &[b, b, d]\\
(2, 3, 5) &c/(a^2d^2) &[a, a, d]\\
(1, 4, 5) &1/(b^2d) &[b, b, c]\\
(2, 4, 5) &(-1)/(a^2d) &[a, a, c]\\
(3, 4, 5) &1/(a^2d) &[a, a, b]\\
(1, 2, 6) &(-1)/(bd^2)& [c, c, d]\\
(1, 3, 6) &c/(b^2d^2) &[b, b, d]\\
(2, 3, 6) &(-c)/(a^2d^2)& [a, a, d]\\
(1, 4, 6) &(-1)/(b^2d) &[b, b, c]\\
(2, 4, 6) &1/(a^2d) &[a, a, c]\\
(3, 4, 6) &(-1)/(a^2d) & [a, a, b]\\
(1, 5, 6) &e/(b^2d^2) &[b, b, c]\\
(2, 5, 6) &(-e)/(a^2d^2)& [a, a, c]\\
(3, 5, 6) &e/(a^2d^2) &[a, a, b]\\
(4, 5, 6) &(-e)/(a^2c^2)& [a, a, b]
\end{array}
$$


\subsection{The skew-symmetric matrix of Lemma \ref{lem:x[abc]iszeroConditionally}}
\label{App:lem:x[abc]iszeroConditionally}
The $6\times 6$ skew-symmetric matrix 
$$A = \left(
\begin{matrix}
 0 & d  &d  &d  &d  &d\\
-d  &0  &d  &d  &d  &d\\
-d &-d  &0  &d  &d  &d\\
-d &-d &-d  &0  &a  &b\\
-d &-d &-d &-a  &0  &c\\
-d &-d &-d &-b &-c  &0
\end{matrix}
\right)$$
has Pfaffian 
$$\Pf(A) = (a - b + c)d^2,$$
and for $1\leq i<j<k\leq 6$, the values of $(-1)^{i+j+k}\frac{\Pf(A)}{\Pf(A^{\wedge}_{ij})\Pf(A^{\wedge}_{ik})
\Pf(A^{\wedge}_{jk})}$ and $\left[A_{\widehat{ijk}}\right]$ are as follows.
$$
\renewcommand\arraystretch{1.5}
\begin{array}{c|c|c}
(i,j,k) & \displaystyle{(-1)^{i+j+k}\frac{\Pf(A)}{\Pf(A^{\wedge}_{ij})\Pf(A^{\wedge}_{ik})
\Pf(A^{\wedge}_{jk})}} &  \left[A_{\widehat{ijk}}\right]\\
&\\
\hline
(1, 2, 3) & 1/d(a - b + c)^2& [a, b, c]\\
(1, 2, 4) &(-1)/(c^2d) &[d, d, c]\\
(1, 3, 4)& 1/(c^2d) &[d, d, c]\\
(2, 3, 4) &(-1)/(c^2d)& [d, d, c]\\
(1, 2, 5) &1/(b^2d) &[d, d, b]\\
(1, 3, 5) &(-1)/(b^2d)& [d, d, b]\\
(2, 3, 5) &1/(b^2d)& [d, d, b]\\
(1, 4, 5) &(a - b + c)/(bcd^2) &[d, d, d]\\
(2, 4, 5) &(-a + b - c)/(bcd^2) &[d, d, d]\\
(3, 4, 5) &(a - b + c)/(bcd^2) &[d, d, d]\\
(1, 2, 6) &(-1)/(a^2d) &[d, d, a]\\
(1, 3, 6) &1/(a^2d) &[d, d, a]\\
(2, 3, 6) &(-1)/(a^2d)& [d, d, a]\\
(1, 4, 6) &(-a + b - c)/(acd^2) &[d, d, d]\\
(2, 4, 6) &(a - b + c)/(acd^2) &[d, d, d]\\
(3, 4, 6) &(-a + b - c)/(acd^2) &[d, d, d]\\
(1, 5, 6) &(a - b + c)/(abd^2) &[d, d, d]\\
(2, 5, 6) &(-a + b - c)/(abd^2) &[d, d, d]\\
(3, 5, 6) &(a - b + c)/(abd^2) &[d, d, d]\\
(4, 5, 6) &(-a + b - c)/d^4& [d, d, d]
\end{array}
$$


\subsection{The skew-symmetric matrix of Lemma \ref{lem:x[abc]iszero}}
\label{App:lem:x[abc]iszero}
The $6\times 6$ skew-symmetric matrix 
$$A = \left(
\begin{matrix}
 0 & a  &b & d & d & d\\
-a  &0  &b & e & e & e\\
-b &-b  &0 & f & f & f\\
-d &-e &-f & 0 & a & b\\
-d &-e &-f &-a & 0 & b\\
-d &-e &-f &-b &-b & 0
\end{matrix}
\right)$$
has Pfaffian 
$$\Pf(A) = a(bd - be + af),$$
and for $1\leq i<j<k\leq 6$, the values of $(-1)^{i+j+k}\frac{\Pf(A)}{\Pf(A^{\wedge}_{ij})\Pf(A^{\wedge}_{ik})
\Pf(A^{\wedge}_{jk})}$ and $\left[A_{\widehat{ijk}}\right]$ are as follows.
$$
\renewcommand\arraystretch{1.5}
\begin{array}{c|c|c}
(i,j,k) & \displaystyle{(-1)^{i+j+k}\frac{\Pf(A)}{\Pf(A^{\wedge}_{ij})\Pf(A^{\wedge}_{ik})
\Pf(A^{\wedge}_{jk})}} &  \left[A_{\widehat{ijk}}\right]\\
&\\
\hline
(1, 2, 3)& (bd - be + af)/(a^2def) &[a, b, b]\\
(1, 2, 4) &(-bd + be - af)/(b^4f) &[f, f, b]\\
(1, 3, 4) &(bd - be + af)/(ab^3e)& [e, e, b]\\
(2, 3, 4) &(-bd + be - af)/(ab^3d) &[d, d, b]\\
(1, 2, 5) &(bd - be + af)/(b^4f) &[f, f, b]\\
(1, 3, 5) &(-bd + be - af)/(ab^3e) &[e, e, b]\\
(2, 3, 5) &(bd - be + af)/(ab^3d) &[d, d, b]\\
(1, 4, 5) &a/b^4 &[b, e, f]\\
(2, 4, 5) &(-a)/b^4& [b, d, f]\\
(3, 4, 5) &1/(ab^2)& [a, d, e]\\
(1, 2, 6) &(-bd + be - af)/(a^2b^2f) &[f, f, a]\\
(1, 3, 6) &(bd - be + af)/(a^3be) &[e, e, a]\\
(2, 3, 6) &(-bd + be - af)/(a^3bd) &[d, d, a]\\
(1, 4, 6) &(-1)/b^3 &[b, e, f]\\
(2, 4, 6) &1/b^3 &[b, d, f]\\
(3, 4, 6) &(-1)/(a^2b) &[a, d, e]\\
(1, 5, 6) &1/b^3 &[b, e, f]\\
(2, 5, 6) &(-1)/b^3& [b, d, f]\\
(3, 5, 6) &1/(a^2b)& [a, d, e]\\
(4, 5, 6) & (-a)/(bd  - be + af)^2 &[a, b, b]\\
&
\end{array}
$$
\bibliographystyle{plain}

\end{document}